\RequirePackage{fix-cm}
\documentclass[smallextended]{svjour3}       
\smartqed  
\usepackage{graphicx}
\usepackage{hyperref}
\usepackage{amsmath}
\usepackage{amssymb}
\usepackage{mathrsfs}
\usepackage{mathtools}
\usepackage{algorithm}
\usepackage{algpseudocode}
\addtocounter{MaxMatrixCols}{10}
\numberwithin{equation}{section}  
\usepackage{booktabs}     
\usepackage{anysize}
\newtheorem{myDef}{Definition} 
\usepackage{color}

\begin{document}

\title{Componentwise perturbation analysis for the generalized Schur decomposition \thanks{The work is supported by the National Natural Science Foundation of China (No. 11671060, 11771099), the Natural Science Foundation of Chongqing, China (No. cstc2019jcyj-msxmX0267) and the Innovation Program of Shanghai Municipal Education Commission.}
}


\author{Guihua Zhang        \and
        Hanyu Li  \and Yimin Wei 
}


\institute{ Guihua Zhang and Hanyu Li\at
              College of Mathematics and Statistics, Chongqing University, Chongqing, 401331, P.R. China \\
              \email{lihy.hy@gmail.com or hyli@cqu.edu.cn}           
           \and
           Yimin Wei \at
             School of Mathematical Sciences and Shanghai Key Laboratory of Contemporary Applied Mathematics, Fudan University, Shanghai, 200433, P.R. China\\
             \email{ymwei@fudan.edu.cn or yimin.wei@gmail.com} 
}

\date{Received: date / Accepted: date}

\maketitle

\begin{abstract}
By defining two important terms called basic perturbation vectors  and obtaining their linear bounds, we obtain the linear componentwise perturbation bounds for unitary factors and upper triangular factors of the generalized Schur decomposition. The perturbation bounds for the diagonal elements of the upper triangular factors and the generalized invariant subspace are also derived. From the former, we present an upper bound and a condition number of the generalized eigenvalue. Furthermore, with
	numerical iterative method, the nonlinear componentwise perturbation bounds of the generalized Schur decomposition are also provided. Numerical examples are given to test the obtained bounds. Among them, we compare our upper bound and condition number of the generalized eigenvalue with their counterparts given in the literature. Numerical results show that they are very close to each other but our results don't contain the information on the left and right generalized eigenvectors.
\end{abstract}

\section{Introduction}
\label{sec:int and no}	
As we know, for a complex square matrix $A\in \mathbb{C}^{n\times n}$, there exists a unitary matrix $U\in \mathbb{C}^{n\times n}$ and an upper triangular matrix $T\in \mathbb{C}^{n\times n}$ such that
\begin{align*}
		A=UTU^H.
	\end{align*}
This decomposition is called the Schur decomposition of the matrix $A$, which is an important and effective tool in numerical linear algebra (e.g.,\cite{golub2013matrix,stewart1990matrix,laub1979schur,zhou2007solutions}).

For the matrix pair $A\in \mathbb{C}^{n \times n}$ and $B\in \mathbb{C}^{n \times n}$, the generalized Schur decomposition is
\begin{equation}
	A=UTV^H,B=URV^H,
	\label{eq:sec1.1}
\end{equation}
where $T=(t_{ij})\in \mathbb{C}^{n\times n}$ and $R=(r_{ij})\in \mathbb{C}^{n \times n}$ are upper triangular matrices, and $U$ and $V$ are unitary matrices. Throughout this paper, we assume that the matrix pair $(A,B)$ is regular, that is, there exists $\lambda\in C$ satisfying det$(A-\lambda B)\not\equiv0$. Then the pairs $(t_{ii},r_{ii})$ or $\lambda_{i}=t_{ii}/r_{ii}$, where $t_{ii}$ and $r_{ii}$ are the diagonal elements of $T$ and $R$, respectively, are the generalized eigenvalues of the matrix pair $(A,B)$. We assume that the matrix pair $(A,B)$ has distinct eigenvalues throughout this paper. The generalized Schur decomposition also has many important applications in many fields
(e.g.,\cite{golub2013matrix,stewart1990matrix,kagstrom1989generalized,diao2013effective}).

There are many works on the applications, computations, and perturbation analysis for the Schur and generalized Schur decompositions (e.g.,\cite{chen2010perturbation,chen2012perturbation,petkov2021componentwise,konstantinov1994nonlocal,sun1995perturbation}). Here, we mainly focus on the perturbation analysis. In 1994, Konstantinov et al.\cite{konstantinov1994nonlocal} presented the perturbation bounds of the Schur decomposition by using the operator splitting, Lyapunov majorants and fixed point theorem. Soon afterwards, based on these techniques, Sun \cite{sun1995perturbation} presented the perturbation bounds of the generalized Schur decomposition. Later, some authors considered the perturbation analysis for the periodic Schur decomposition \cite{chen2010perturbation} and the antitriangular Schur decomposition \cite{chen2012perturbation}. All the above perturbation analysis belong to the normwise perturbation analysis, because they are obtained by using some 
matrix norms. As we know, normwise perturbation bounds cannot measure the abnormal perturbation of individual elements. Motivated by this, recently, Petkov \cite{petkov2021componentwise} presented the componentwise perturbation analysis of the Schur decomposition.

With the technique in \cite{petkov2021componentwise}, in this paper, we consider the componentwise perturbation analysis of the generalized Schur decomposition. One of the most important contribution is that we present an upper bound and a condition number of the generalized eigenvalue. Unlike the corresponding results in \cite{stewart1990matrix,fraysse1998note,higham1998structured}, they don't include the information on the left and right generalized eigenvectors. Numerical results show that our upper bound 
and condition number are close to their counterparts given in \cite{stewart1990matrix,fraysse1998note}. 

The rest of this paper is organized as follows.
In Section \ref{sec:pre}, we introduce some preliminaries which include some notations, the perturbation equations about the generalized Schur decomposition and the definitions of the basic perturbation vectors and their linear bounds. The linear perturbation bounds for $U$ and $V$, $T$ and $R$, diagonal elements of $T$ and  $R$, and generalized invariant subspace are presented in Sections \ref{sec:u and v}, \ref{sec:T and R},  \ref{diagonal}, and \ref{sec:invariant}, respectively. In Section \ref{sec:nonlinear}, we provide the nonlinear perturbation bounds of the generalized Schur decomposition. Finally, some numerical results are shown to test the obtained bounds.

\section{Preliminaries}
\label{sec:pre}
We first introduce some notations used in this paper. For the matrices $A$ and $B$, $A \otimes B$ denotes their Kronecker product. Given a matrix $A\in \mathbb{C}^{n \times n}$, $A^T,\overline{A}$, $A^H$, $|A|$, $\Vert A\Vert_{2}$ and $\Vert A\Vert_{F}$  denote its transpose, conjugate, conjugate transpose, absolute values, spectral norm and Frobenius norm, respectively. In addition, let $I_{n}$ be the $n \times n$ identity matrix,  $A_{i,: } $ be the  $ i $th row of $A$, $A_{:,j } $ be the  $ j $th column of $A$ and $A_{i_1:i_2,j_1:j_2 } $ be the part of $A$ consisting of rows from $i_1$ to $i_2$ and columns from $j_1$ to $j_2$.

For the matrix $A=\begin{bmatrix}
	a_{1} , a_{2} , \dots , a_{n}
\end{bmatrix}=(a_{ij}) \in \mathbb{C}^{n \times n}$, we denote the vector ${\rm vec}(A)$ by ${\rm vec}(A)=\begin{bmatrix}
	a_{1}^T , a_{2}^T , \dots , a_{n}^T
\end{bmatrix}^T$, the vector of the last $i$ elements of $a_j$ by $a_j^{[i]}$ and the upper triangular part of $A$ by ${\rm triu}(A)$. It is easy to check that 
\begin{equation}
	\vert {\rm triu}(AB)\vert={\rm triu}(\vert  AB\vert) \preceq {\rm triu}(\vert  A\vert\vert  B\vert),
	\label{triu}
\end{equation}
where $A\in \mathbb{C}^{n \times n}$ and $B\in \mathbb{C}^{n \times n}$.
Meanwhile,
\begin{equation}
	{\rm vec}(A^T)=\Pi_{n}{\rm vec}(A),
	\label{permutation}
\end{equation} 
where $\Pi_{n}$ is the vec-permutation matrix.
For simplicity of explicit expressions, we consider the notations in \cite{li2015improved},
\begin{align*}
	{\rm slvec}(A)=\begin{bmatrix}
		a_1^{[n-1]} \\ a_2^{[n-2]} \\ \vdots \\ a_{n-1}^{[1]} \\
	\end{bmatrix}\in \mathbb{C}^{\nu},
\end{align*}
\begin{align*}
	M_{{\rm slvec}}=\begin{bmatrix}
		{\rm diag}(J_1,J_2,\cdots,J_{n-1}) , 0_{\nu\times n}
	\end{bmatrix}\in \mathbb{C}^{\nu\times n^2},
	J_i=\begin{bmatrix}
			0_{(n-i)\times i},I_{n-i}
		\end{bmatrix}\in \mathbb{C}^{(n-i)\times n},
\end{align*}
where $\nu=\frac{n(n-1)}{2}$ and $0_{m\times n}$ is the $m\times n$ zero matrix. It is easy to show that  
\begin{equation}
	{\rm slvec}(A)=M_{{\rm slvec}}{\rm vec}(A).
	\label{lu} 
\end{equation}

Assume that the perturbed matrix pair $(\widetilde{A},\widetilde{B})$ has the following generalized Schur decomposition
\begin{equation}
	\widetilde{A}=A+\Delta A=\widetilde{U}\widetilde{T}\widetilde{V}^H,\widetilde{B}=B+\Delta B=\widetilde{U}\widetilde{R}\widetilde{V}^H,
	\label{eq:sec2.1}
\end{equation}
where
\begin{align*}
	\widetilde{U}=U+\Delta U,
	\widetilde{V}=V+\Delta V, \\
	\widetilde{T}=T+\Delta T,
	\widetilde{R}=R+\Delta R.
\end{align*}
To analyze the perturbation bounds of the generalized Schur decomposition conveniently, we first define two important matrices, that is,
\begin{equation}
	\Delta W=U^H\Delta U,\Delta K=V^H\Delta V.
	\label{eq:sec2.6}
\end{equation}
Further, we define the following two vectors 
\begin{equation}
	\begin{cases}
		x={\rm slvec}(\Delta W),\\
		y={\rm slvec}(\Delta K) ,
	\end{cases}
	\label{eq:sec2.7}
\end{equation}
and call them the basic perturbation vectors. Considering (\ref{lu}), (\ref{eq:sec2.7}) can be written as
\begin{align*}
	\begin{cases}
		x=M_{{\rm slvec}}{\rm vec}(\Delta W),\\
		y=M_{{\rm slvec}}{\rm vec}(\Delta K).
	\end{cases}
\end{align*}
In the following, we will obtain the linear bounds of $x$ and $y$. To this end, let
\begin{align*}
	U&=\begin{bmatrix} u_{1},u_{2}, \dots , u_{n}\end{bmatrix},\Delta U=\begin{bmatrix}
		\Delta u_{1},\Delta u_{2},\dots,\Delta u_{n}
	\end{bmatrix}, \\
   V&=\begin{bmatrix}v_{1} , v_{2} , \dots , v_{n}\end{bmatrix},\Delta V=\begin{bmatrix}
   	\Delta v_{1},\Delta v_{2},\dots,\Delta v_{n}
   \end{bmatrix},
\end{align*}
where $u_{j}$ and $v_{j}$ are the $j$th column vectors of $U$ and $V$, respectively.
Thus, based on the relationship between $U$ and $\widetilde U$, and the relationship between $V$ and $\widetilde V$, we have 
\begin{align*}
	\widetilde{U}&=\begin{bmatrix}
		\widetilde{u}_{1},\widetilde{u}_{2},\dots,\widetilde{u}_{n}
	\end{bmatrix},\widetilde{u}_{j}=u_{j}+\Delta u_{j}, \\
\widetilde{V}&=\begin{bmatrix}
	\widetilde{v}_{1},\widetilde{v}_{2},\dots, \widetilde{v}_{n}
\end{bmatrix},\widetilde{v}_{j}=v_{j}+\Delta v_{j}.
\end{align*}
Since $T,R,\widetilde{T}$ and $\widetilde{R}$ are upper triangular matrices, considering (\ref{eq:sec1.1}) and (\ref{eq:sec2.1}), it follows that
\begin{align*}
	\widetilde{u}_{i}^H(A+\Delta A)\widetilde{v}_{j}&=u_{i}^HAv_{j}=0,\\
	\widetilde{u}_{i}^H(B+\Delta B)\widetilde{v}_{j}&=u_{i}^HBv_{j}=0,
\end{align*} 
where $1\leq j<i\leq n$.
Hence
\begin{equation}
	\begin{cases}
		u_{i}^HA\Delta v_{j}+\Delta u_{i}^HAv_{j}+\Delta u_{i}^HA\Delta v_{j}=-\widetilde{u}_{i}^H\Delta A\widetilde{v}_{j}, \\
		u_{i}^HB\Delta v_{j}+\Delta u_{i}^HBv_{j}+\Delta u_{i}^HB\Delta v_{j}=-\widetilde{u}_{i}^H\Delta B\widetilde{v}_{j}.	
	\end{cases}
	\label{eq:sec2.2}
\end{equation}
Note that the componentwise representation of (\ref{eq:sec1.1}) can be written by
\begin{equation}
	\begin{cases}
		u_{i}^HA=\sum_{k=i}^n t_{ik}v_{k}^H, \\
		Av_{j}=\sum_{k=1}^j t_{kj}u_{k},
	\end{cases}
	\begin{cases}
		u_{i}^HB=\sum_{k=i}^n r_{ik}v_{k}^H, \\
		Bv_{j}=\sum_{k=1}^j r_{kj}u_{k},
	\end{cases}
	\label{eq:sec2.3}
\end{equation}
where $1\leq j<i\leq n$.
Thus by combining (\ref{eq:sec2.2}) with (\ref{eq:sec2.3}), we get 
\begin{equation}
	\begin{cases}
		\sum_{k=i}^n t_{ik}v_{k}^H\Delta v_{j}+\sum_{k=1}^j t_{kj}\Delta u_{i}^Hu_{k}+\Delta u_{i}^HA\Delta v_{j}=-\widetilde{u}_{i}^H\Delta A \widetilde{v}_{j}, \\
		\sum_{k=i}^n r_{ik}v_{k}^H\Delta v_{j}+\sum_{k=1}^j r_{kj}\Delta u_{i}^Hu_{k}+\Delta u_{i}^HB\Delta v_{j}=-\widetilde{u}_{i}^H\Delta B \widetilde{v}_{j},
	\end{cases}	
	\label{eq:sec2.4}
\end{equation}
where $1\leq j<i\leq n$.
Because of the unitarity of $U$ and $\widetilde U$, it follows that
\begin{align*}
	U^H\Delta U=-\Delta U^HU-\Delta U^H\Delta U,
\end{align*}
whose componentwise representation can be written by
\begin{equation}
	\Delta u_{i}^Hu_{j}=-u_{i}^H\Delta u_{j}-\Delta u_{i}^H\Delta u_{j}.
	\label{delta uiuj}
\end{equation}
Substitute (\ref{delta uiuj}) into (\ref{eq:sec2.4})  getting the following equations
\begin{equation}
	\begin{cases}
		\sum_{k=i}^n t_{ik}v_{k}^H\Delta v_{j}-\sum_{k=1}^j t_{kj}u_{i}^H\Delta u_{k}-\sum_{k=1}^j t_{kj}\Delta u_{i}^H\Delta u_{k}+\Delta u_{i}^HA\Delta v_{j}=-\widetilde{u}_{i}^H\Delta A \widetilde{v}_{j}, \\
		\sum_{k=i}^n r_{ik}v_{k}^H\Delta v_{j}-\sum_{k=1}^j r_{kj}u_{i}^H\Delta u_{k}-\sum_{k=1}^j r_{kj}\Delta u_{i}^H\Delta u_{k}+\Delta u_{i}^HB\Delta v_{j}=-\widetilde{u}_{i}^H\Delta B \widetilde{v}_{j},
	\end{cases}	
	\label{eq:sec2.5}
\end{equation}
where $1\leq j<i\leq n$.

The (\ref{eq:sec2.5}) is very important to find the linear bounds of the basic perturbation vectors $x$ and $y$.
To continue, we need to write it as a matrix-vector equation. Some notations are necessary.
\begin{align*}
	F&=-\widetilde U^H\Delta A\widetilde V,f={\rm slvec}(F)=M_{{\rm slvec}}{\rm vec}(F),\\
	G&=-\widetilde U^H\Delta B\widetilde V,g={\rm slvec}(G)=M_{{\rm slvec}}{\rm vec}(G), \\
	\Delta^x_{l}&=\sum_{k=1}^j t_{kj}\Delta u_{i}^H\Delta u_{k}-\Delta u_{i}^HA\Delta v_{j}, \\
	\Delta^y_{l}&=\sum_{k=1}^j r_{kj}\Delta u_{i}^H\Delta u_{k}-\Delta u_{i}^HB\Delta v_{j},
\end{align*}
where $1\leq j<i\leq n$ and $l=i+(j-1)n-\frac{j(j+1)}{2}$. 
Then, (\ref{eq:sec2.5}) can be written as
\begin{equation}
	\begin{cases}
			-L_{T,X}x+L_{T,Y}y=f+\Delta^x, \\
			-L_{R,X}x+L_{R,Y}y=g+\Delta^y,
	\end{cases}
	\label{eq:sec2.8}
\end{equation}
where the vectors $\Delta^x\in \mathbb{C}^{\nu}$ and $\Delta^y\in \mathbb{C}^{\nu}$ consist of $\Delta^x_{l}$ and $\Delta^y_{l}$, respectively, and
\begin{align*}
	L_{T,X}&=M_{\rm slvec}(T^T\otimes I_n)M_{\rm slvec}^T\in \mathbb{C}^{\nu\times\nu}, L_{T,Y}=M_{\rm slvec}(I_n \otimes T)M_{\rm slvec}^T\in \mathbb{C}^{\nu\times\nu}, \\
	L_{R,X}&=M_{\rm slvec}(R^T\otimes I_n)M_{\rm slvec}^T\in \mathbb{C}^{\nu\times\nu}, L_{R,Y}=M_{\rm slvec}(I_n \otimes R)M_{\rm slvec}^T\in \mathbb{C}^{\nu\times\nu}.
\end{align*}
Notice that $L_{T,X}$ and $L_{R,X}$ are block lower triangular matrices, and $L_{T,Y}$ and $L_{R,Y}$ are block diagonal matrices. For example, for $n=5$, the matrices $L_{T,X}$ and $L_{T,Y}$ have the following forms
\begin{align*}
	L_{T,X}&=\begin{bmatrix}
		t_{11} & 0 & 0 & 0 & 0 & 0 & 0 & 0 & 0 & 0\\
		0 & t_{11} & 0 & 0 & 0 & 0 & 0 & 0 & 0 & 0\\
		0 & 0 & t_{11} & 0 & 0 & 0 & 0 & 0 & 0 & 0\\
		0 & 0 & 0 & t_{11} & 0 & 0 & 0 & 0 & 0 & 0\\
		0 & t_{12} & 0 & 0 & t_{22} & 0 & 0 & 0 & 0 & 0\\
		0 & 0 & t_{12} & 0 & 0 & t_{22} & 0 & 0 & 0 & 0\\
		0 & 0 & 0 & t_{12} & 0 & 0 & t_{22} & 0 & 0 & 0\\
		0 & 0 & t_{13} & 0 & 0 & t_{23} & 0 & t_{33} & 0 & 0\\
		0 & 0 & 0 & t_{13} & 0 & 0 & t_{23} & 0 & t_{33} & 0\\
		0 & 0 & 0 & t_{14} & 0 & 0 & t_{24} & 0 & t_{34} & t_{44}\\
	\end{bmatrix},\\
	L_{T,Y}&=\begin{bmatrix}
		t_{22} & t_{23} & t_{24} & t_{25} & 0 & 0 & 0 & 0 & 0 & 0\\
		0 & t_{33} & t_{34} & t_{35} & 0 & 0 & 0 & 0 & 0 & 0\\
		0 & 0 & t_{44} & t_{45} & 0 & 0 & 0 & 0 & 0 & 0\\
		0 & 0 & 0 & t_{55} & 0 & 0 & 0 & 0 & 0 & 0\\
		0 & 0 & 0 & 0 & t_{33} & t_{34} & t_{35} & 0 & 0 & 0\\
		0 & 0 & 0 & 0 & 0 & t_{44} & t_{45} & 0 & 0 & 0\\
		0 & 0 & 0 & 0 & 0 & 0 & t_{55} & 0 & 0 & 0\\
		0 & 0 & 0 & 0 & 0 & 0 & 0 & t_{44} & t_{45} & 0\\
		0 & 0 & 0 & 0 & 0 & 0 & 0 & 0 & t_{55} & 0\\
		0 & 0 & 0 & 0 & 0 & 0 & 0 & 0 & 0 & t_{55}\\
	\end{bmatrix}.
\end{align*}
Let
\begin{equation}
	L=\begin{bmatrix}
		-L_{T,X} & L_{T,Y}\\
		-L_{R,X} & L_{R,Y}
	\end{bmatrix}=\begin{bmatrix}
		-M_{\rm slvec} & 0 \\ 0 & -M_{\rm slvec}
	\end{bmatrix}\begin{bmatrix}
		T^T\otimes I_n & I_n\otimes T \\
		R^T\otimes I_n & I_n\otimes R
	\end{bmatrix}\begin{bmatrix}
		M_{\rm slvec}^T & 0 \\ 0 & M_{\rm slvec}^T
	\end{bmatrix}\in \mathbb{C}^{2\nu\times2\nu}.
	\label{L}
\end{equation}
Then (\ref{eq:sec2.8}) can be simplified as 
\begin{equation}
	L\begin{bmatrix}
			x \\ y
		\end{bmatrix}=\begin{bmatrix}
		f \\ g
	\end{bmatrix}+\begin{bmatrix}
			\Delta^x \\ \Delta^y
		\end{bmatrix}.
	\label{eq:sec2.9}
\end{equation} 
Note that $L$ is nonsingular \cite{sun1995perturbation}. Thus, neglecting the second order terms, we obtain the linear approximation of $x$ and $y$,
\begin{equation}
	\begin{bmatrix}
		x\\y
	\end{bmatrix}_{lin}=L^{-1}\begin{bmatrix}
		f\\g
	\end{bmatrix},
	\label{LNI1}
\end{equation}	
which implies that the $l$th component of 
$\begin{bmatrix}
	x\\y
\end{bmatrix}_{lin}$
can be written as
\begin{equation}
	\begin{bmatrix}
		x\\y
	\end{bmatrix}_{lin,l}=L^{-1}_{l,:}\begin{bmatrix}
		f\\g
	\end{bmatrix},
	\label{LNI}
\end{equation}
where $l=1,2,\dots,2\nu$.
Since
\begin{align*}
	\begin{cases}
		\Vert f\Vert_{2}\leq \Vert \Delta A\Vert_{F}, \\
		\Vert g\Vert_{2}\leq \Vert \Delta B\Vert_{F},
	\end{cases}
\end{align*}
we have the linear bound of the $l$th component 
\begin{equation}
	\left\vert \begin{bmatrix}
		x\\y
	\end{bmatrix}_{lin,l} \right\vert\leq \Vert L^{-1}_{l,:}\Vert_{2}\left\Vert\begin{bmatrix}
		\Delta A \\ \Delta B
	\end{bmatrix}\right\Vert_{F},
\end{equation}
and hence the linear bounds of the $j$th component of vectors $x$ and $y$
\begin{equation}
	\vert x_{lin,j}\vert \leq \Vert L^{-1}_{j,:}\Vert_{2}\left\Vert\begin{bmatrix}
		\Delta A \\ \Delta B
	\end{bmatrix}\right\Vert_{F},
	\vert y_{lin,j}\vert \leq \Vert L^{-1}_{j+\nu,:}\Vert_{2}\left\Vert\begin{bmatrix}
		\Delta A \\ \Delta B
	\end{bmatrix}\right\Vert_{F},
	\label{eq:sec2.10}
\end{equation}
where $j=1,2,\dots,\nu$.

The above discussions are summarized in the following theorem.
\begin{theorem} \label{theorem1}
	Let $(A,B)$ and $(\widetilde{A},\widetilde{B})$ have the generalized Schur decompositions in  (\ref{eq:sec1.1}) and (\ref{eq:sec2.1}), respectively. 
	For $x$ and $y$ defined in (\ref{eq:sec2.7}), the linear bounds (\ref{eq:sec2.10}) hold.
\end{theorem}
In the following sections, we will use Theorem \ref{theorem1}, i.e., the bounds (\ref{eq:sec2.10}), to investigate the componentwise perturbation bounds of the generalized Schur decomposition.

\section{Linear perturbation bounds for $U$ and $V$} \label{sec:u and v}
We first rewrite the matrices $\Delta W$ and $\Delta K$ in (\ref{eq:sec2.6}) as the sum of matrices whose elements only contain the first or second order terms.
To do this, note that $\widetilde{U}$ and $U$ are unitary matrices, then 
\[\Delta u_{i}^Hu_{j}=-u_{i}^H\Delta u_{j}-\Delta u_{i}^H\Delta u_{j}\]
or
\[\overline{ u_{j}^H\Delta u_{i}}=-u_{i}^H\Delta u_{j}-\Delta u_{i}^H\Delta u_{j},\]
where $\overline{ u_{j}^H\Delta u_{i}}$ with $1\leq i<j\leq n$ is the conjugate of the element $x_{l}$, that is $\overline{ u_{j}^H\Delta u_{i}}=\overline{x}_l$ with $l=j+(i-1)n-\frac{i(i+1)}{2}$. For the  matrices $\widetilde{V}$ and $V$, similar results hold. Thus, $\Delta W$ and $\Delta K$ can be written as 
\begin{equation}
	\begin{cases}
		\Delta W=\Delta W_{1}+\Delta W_{2}-\Delta W_{3}, \\
		\Delta K=\Delta K_{1}+\Delta K_{2}-\Delta K_{3} ,
	\end{cases}
	\label{eq:sec3.6}
\end{equation}
where 
\begin{align*}
    \Delta W_{1}=\begin{bmatrix}
	0 & -\overline x_{1} & -\overline x_{2} & \dots & -\overline x_{n-1} \\
	x_{1} & 0 & -\overline x_{1} & \dots & -\overline x_{2n-3} \\
	x_{2} & x_{n} & 0 & \dots & -\overline x_{3n-6}  \\
	\vdots & \vdots & \vdots & \ddots & \vdots\\ 
	x_{n-1} & x_{2n-3} & x_{3n-6} & \dots & 0
\end{bmatrix}\in \mathbb{C}^{n \times n},
\end{align*}
\begin{equation}
	\Delta W_{2}={\rm diag}(u_{1}^H\Delta u_{1},u_{2}^H\Delta u_{2},\dots,u_{n}^H\Delta u_{n})\in \mathbb{C}^{n \times n},
	\label{delta w2}
\end{equation}
\begin{equation}
	\Delta W_{3}=\begin{bmatrix}
		0 & \Delta u_{1}^H\Delta u_{2} & \Delta u_{1}^H\Delta u_{3} & \dots & \Delta u_{1}^H\Delta u_{n} \\
		0 & 0 & \Delta u_{2}^H\Delta u_{3} & \dots & \Delta u_{2}^H\Delta u_{n} \\
		0 & 0 & 0 & \dots & \Delta u_{3}^H\Delta u_{n} \\
		\vdots & \vdots & \vdots & \ddots & \vdots\\
		0 & 0 & 0 & \dots & 0 
	\end{bmatrix}\in \mathbb{C}^{n \times n},
	\label{delta w3}
\end{equation}
and $\Delta K_{1},\Delta K_{2}$ and $\Delta K_{3}$ are the same as $\Delta W_{1},\Delta W_{2}$ and $\Delta W_{3}$, respectively, except that $x_{i}$ is replaced by $y_{i}$ and $\Delta u_{j}$ is replaced by $\Delta v_{j}$.
Obviously, the factors $\Delta W_{3}$ and $\Delta K_{3}$ are comprised of the high order terms.

To obtain the desired bounds, we assume that the perturbed unitary matrices $\widetilde{U}$ and $\widetilde{V}$ are such that the imaginary parts of diagonal elements of $\Delta W_{2}$ and $\Delta K_{2}$ are zero. This assumption also appears in \cite{petkov2021componentwise} and is common in perturbation analysis for Schur and generalized Schur decompositions (e.g.,\cite{konstantinov1994nonlocal,sun1995perturbation}). Then it's easy to show that $u_{i}^H\Delta u_{i}=-\frac{\Delta u_{i}^H\Delta u_{i}}{2}$ and $v_{i}^H\Delta v_{i}=-\frac{\Delta v_{i}^H\Delta v_{i}}{2}$. Hence the elements of $\Delta W_{2}$ and $\Delta K_{2}$ are also the high order terms.
Thus, together with (\ref{eq:sec2.10}), we can obtain the linear bounds of $\Delta W$ and $\Delta K$, i.e., $\vert \widehat{\Delta W} \vert$ and $\vert \widehat{\Delta K} \vert$, which are given as follows:
\begin{equation}
	\vert \widehat{\Delta W} \vert =\begin{bmatrix}
		0 & \vert x_{1} \vert & \vert x_{2}\vert & \dots & \vert x_{n-1} \vert \\
		\vert x_{1} \vert & 0 & \vert x_{n} \vert & \dots & \vert x_{2n-3} \vert \\
		\vert x_{2} \vert & \vert x_{n} \vert & 0 & \dots & \vert x_{3n-6} \vert \\
		\vdots & \vdots & \vdots & \ddots & \vdots\\
		\vert x_{n-1} \vert & \vert x_{2n-3} \vert &
		\vert x_{3n-6} \vert & \dots & 0
	\end{bmatrix},
	\label{Delta W1}
\end{equation}
\begin{equation}
	\vert \widehat{\Delta K} \vert=\begin{bmatrix}
		0 & \vert y_{1} \vert & \vert y_{2}\vert & \dots & \vert y_{n-1} \vert \\
		\vert y_{1} \vert & 0 & \vert y_{n} \vert & \dots & \vert y_{2n-3} \vert \\
		\vert y_{2} \vert & \vert y_{n} \vert & 0 & \dots & \vert y_{3n-6} \vert \\
		\vdots & \vdots & \vdots & \ddots & \vdots\\
		\vert y_{n-1} \vert & \vert y_{2n-3} \vert &
		\vert y_{3n-6} \vert & \dots & 0
	\end{bmatrix}.
	\label{Delta K1}
\end{equation}
Since $U$ and $V$ are unitary matrices, the linear perturbation bounds of $U$ and $V$ follow as
\begin{equation}
	\begin{cases}
		\vert \Delta U \vert \preceq  \vert U\vert\vert U^H\Delta U\vert =\vert U\vert\vert \widehat{\Delta W}\vert, \\
		\vert \Delta V \vert \preceq \vert V\vert\vert V^H\Delta V\vert =\vert V\vert\vert \widehat{\Delta K}\vert .
	\end{cases}
	\label{eq:sec3.7}
\end{equation}

In summary, we have the following theorem.
\begin{theorem} \label{theorem3.1}
	Let $(A,B)$ and $(\widetilde{A},\widetilde{B})$ have the generalized Schur decompositions in (\ref{eq:sec1.1}) and (\ref{eq:sec2.1}), respectively. 
	Then the linear perturbation bounds of $U$ and $V$ given in (\ref{eq:sec3.7}) hold. 
\end{theorem}


\section{Linear perturbation bounds for $T$ and $R$}
\label{sec:T and R} 
\begin{theorem} \label{theorem4.1}
	Let $(A,B)$ and $(\widetilde{A},\widetilde{B})$ have the generalized Schur decompositions in (\ref{eq:sec1.1}) and (\ref{eq:sec2.1}), respectively. Then the linear perturbation bounds for $T$ and $R$ are 
	\begin{equation}
		\begin{split}
			\vert \widehat{\Delta T}\vert ={\rm triu}&(\vert U^H\vert (\vert A\vert+\vert \Delta  A\vert)\vert V\vert \vert \widehat{\Delta K}\vert+\vert \widehat{\Delta W}^H\vert\vert U^H\vert (\vert A\vert+\vert \Delta  A\vert)\vert V\vert \\
			&+\vert \widehat{\Delta W}^H\vert\vert U^H\vert (\vert A\vert+\vert \Delta  A\vert)\vert V\vert\vert \widehat{\Delta K}\vert+\vert U^H\vert\vert \Delta  A\vert\vert V\vert),
		\end{split}	
		\label{4.3}
	\end{equation}
	and
	\begin{equation}
		\begin{split}
			\vert \widehat{\Delta R}\vert ={\rm triu}&(\vert U^H\vert (\vert B\vert+\vert \Delta  B\vert)\vert V\vert \vert \widehat{\Delta K}\vert+\vert \widehat{\Delta W}^H\vert\vert U^H\vert (\vert B\vert+\vert \Delta  B\vert)\vert V\vert \\
			&+\vert \widehat{\Delta W}^H\vert\vert U^H\vert (\vert B\vert+\vert \Delta  B\vert)\vert V\vert\vert \widehat{\Delta K}\vert+\vert U^H\vert\vert \Delta  B\vert\vert V\vert),
		\end{split}	
		\label{4.4}
	\end{equation}
	where $\vert \widehat{\Delta W}\vert$ and $\vert \widehat{\Delta K}\vert$ are given in (\ref{Delta W1}) and (\ref{Delta K1}), respectively.
\end{theorem}
\begin{proof}
	From (\ref{eq:sec2.1}), $\Delta T$ can be written as
	\begin{align*}
	    \Delta T=\widetilde{T}-T=\widetilde{U}^H\widetilde{A}\widetilde{V}-U^HAV.
	\end{align*}
	Then
	\begin{equation}
		\Delta T=U^HA\Delta V+\Delta U^HAV+\Delta U^HA\Delta V+U^H\Delta AV+U^H\Delta A\Delta V+\Delta U^H\Delta AV+\Delta U^H\Delta A\Delta V.
		\label{4.1}
	\end{equation}
	Since $\Delta T$ is an upper triangular matrix, applying $\rm triu$ to (\ref{4.1}) implies 
	\begin{align*}
	    \Delta T={\rm triu}(U^HA\Delta V+\Delta U^HAV+\Delta U^HA\Delta V+U^H\Delta AV+U^H\Delta A\Delta V+\Delta U^H\Delta AV+\Delta U^H\Delta A\Delta V).
	\end{align*}
	Combining the aforementioned equation with (\ref{triu}), we have 
	\begin{equation}
		\begin{split}
				\vert \Delta T\vert \preceq {\rm triu}&(\vert U^H\vert\vert A\vert \vert\Delta V\vert+
				\vert \Delta U^H\vert \vert A\vert \vert V\vert+\vert \Delta U^H\vert \vert A\vert \vert \Delta V\vert +\vert U^H\vert \vert \Delta A\vert \vert V\vert \\
				&+\vert U^H\vert \vert \Delta A\vert \vert \Delta V\vert +\vert \Delta U^H\vert \vert \Delta A\vert \vert V\vert +\vert \Delta U^H\vert \vert \Delta A\vert \vert \Delta V\vert ).
		\end{split}	
		\label{4.2}
	\end{equation}
	Considering (\ref{eq:sec3.7}) and (\ref{4.2}), the perturbation bound of $T$, i.e., the bound (\ref{4.3}), holds.  
	Similarly, we can obtain the perturbation bound of $R$, i.e., the bound (\ref{4.4}). 
\end{proof}


\section{Linear perturbation bounds for diagonal elements of $T$ and $R$}
\label{diagonal}
The method is similar to the one of obtaining the linear bounds of $x$ and $y$. That is, by using the componentwise representations of (\ref{eq:sec1.1}) and (\ref{eq:sec2.1}), we can construct matrix-vector equations and then find the desired bounds.
Specifically, noting that the matrices $T,R,\widetilde T$ and $\widetilde R$ are upper triangular matrices, we have
\begin{equation}
	\begin{cases}
		\Delta t_{ij}=\widetilde t_{ij}-t_{ij}=u_{i}^HA\Delta v_{j}+\Delta u_{i}^HAv_{j}+\Delta u_{i}^HA\Delta v_{j}+\widetilde u_{i}^H\Delta A\widetilde v_{j} ,\\
		\Delta r_{ij}=\widetilde r_{ij}-r_{ij}=u_{i}^HB\Delta v_{j}+\Delta u_{i}^HBv_{j}+\Delta u_{i}^HB\Delta v_{j}+\widetilde u_{i}^H\Delta B\widetilde v_{j} ,
	\end{cases}
	\label{eq:sec3.8}
\end{equation}
where $1\leq i\leq j\leq n$. 
Hence, when $i=j$, (\ref{eq:sec3.8}) can be written as 
\begin{equation}
	\begin{cases}
		\Delta t_{ii}=\sum_{k=i}^{n}t_{ik}v_{k}^H\Delta v_{i}-\sum_{k=1}^{i}t_{ki}u_{i}^H\Delta u_{k}-\sum_{k=1}^{i}t_{ki}\Delta u_{i}^H\Delta u_{k}+\Delta u_{i}^HA\Delta v_{i}+\widetilde u_{i}^H\Delta A \widetilde v_{i}, \\
		\Delta r_{ii}=\sum_{k=i}^{n}r_{ik}v_{k}^H\Delta v_{i}-\sum_{k=1}^{i}r_{ki}u_{i}^H\Delta u_{k}-\sum_{k=1}^{i}r_{ki}\Delta u_{i}^H\Delta u_{k}+\Delta u_{i}^HB\Delta v_{i}+\widetilde u_{i}^H\Delta B \widetilde v_{i} ,
	\end{cases}
	\label{eq:sec3.9}
\end{equation}
where $i=1,2,\dots,n$.
Next, we write it as a matrix-vector equation.  Some notations are necessary.
\begin{equation}
	f_{1}=\begin{bmatrix}
		\widetilde u_{1}^H\Delta A\widetilde v_{1}\\
		\widetilde u_{2}^H\Delta A\widetilde v_{2}\\
		\vdots \\
		\widetilde u_{n}^H\Delta A\widetilde v_{n}
	\end{bmatrix},
	g_{1}=\begin{bmatrix}
		\widetilde u_{1}^H\Delta B\widetilde v_{1}\\
		\widetilde u_{2}^H\Delta B\widetilde v_{2}\\
		\vdots \\
		\widetilde u_{n}^H\Delta B\widetilde v_{n}
	\end{bmatrix},\Delta \lambda^t=\begin{bmatrix}
		\Delta t_{11} \\ \Delta t_{22} \\ \vdots \\ \Delta t_{nn}
	\end{bmatrix},
	\Delta \lambda^r=\begin{bmatrix}
		\Delta r_{11} \\ \Delta r_{22} \\ \vdots \\ \Delta r_{nn}
	\end{bmatrix},	
\end{equation}
\begin{equation}
	d_t=\begin{bmatrix}
		t_{11}(v_1^H\Delta v_1-u_1^H\Delta u_1) \\
		t_{22}(v_2^H\Delta v_2-u_2^H\Delta u_2) \\
		\vdots \\
		t_{nn}(v_n^H\Delta v_n-u_n^H\Delta u_n) \\
	\end{bmatrix},
	d_r=\begin{bmatrix}
		r_{11}(v_1^H\Delta v_1-u_1^H\Delta u_1) \\
		r_{22}(v_2^H\Delta v_2-u_2^H\Delta u_2) \\
		\vdots \\
		r_{nn}(v_n^H\Delta v_n-u_n^H\Delta u_n) \\
	\end{bmatrix},\Delta^{dt}=\begin{bmatrix}
		\Delta^{dt}_{1} \\ \Delta^{dt}_{2} \\ \vdots
		\\ \Delta^{dt}_{n}
	\end{bmatrix},
	\Delta^{dr}=\begin{bmatrix}
		\Delta^{dr}_{1} \\ \Delta^{dr}_{2} \\ \vdots
		\\ \Delta^{dr}_{n}
	\end{bmatrix},
	\label{second order diagonal}
\end{equation}	
where
\begin{align*}
    \Delta^{dt}_{i}&=-\sum_{k=1}^{i}t_{ki}\Delta u_{i}^H\Delta u_{k}+\Delta u_{i}^HA\Delta v_{i}, \\
    \Delta^{dr}_{i}&=-\sum_{k=1}^{i}r_{ki}\Delta u_{i}^H\Delta u_{k}+\Delta u_{i}^HB\Delta v_{i}.
\end{align*}
Clearly, the vectors $\Delta \lambda^t$ and $\Delta \lambda^r$ consist of diagonal elements of the matrices $\Delta T$ and $\Delta R$, respectively, and $\Delta^{dt}$,  $\Delta^{dr}$, $d_t$ and $d_r$ contain the second order terms. Thus, considering (\ref{eq:sec3.9}), we get  
\begin{equation}
	\begin{bmatrix}
		\Delta \lambda^t \\ \Delta \lambda^r
	\end{bmatrix}=
	\begin{bmatrix}
		A_{1} & A_{2} \\ B_{1} & B_{2}
	\end{bmatrix}
	\begin{bmatrix}
		x \\ y
	\end{bmatrix}+
	\begin{bmatrix}
		f_{1} \\ g_{1}
	\end{bmatrix}+
	\begin{bmatrix}
		\Delta^{dt} \\ \Delta^{dr}
	\end{bmatrix}+\begin{bmatrix}
		d_t \\ d_r
	\end{bmatrix},
	\label{eq:sec3.10}
\end{equation}
where
\begin{align*}
    A_{1}&=\begin{bmatrix}
	0 & 0 & 0 & \cdots & 0 & 0 & 0 & 0 & \cdots & 0 & 0 & 0 & 0 & \cdots & 0 & \cdots & 0  \\
	-t_{12} & 0 & 0 & \cdots & 0 & 0 & 0 & 0 & \cdots & 0 & 0 & 0 & 0 & \cdots & 0 & \cdots & 0 \\
	0 & -t_{13} & 0 & \cdots & 0 & -t_{23} & 0 & 0 & \cdots & 0 & 0 & 0 & 0 & \cdots & 0 & \cdots & 0  \\
	\vdots & \vdots & \ddots & \cdots & \vdots &
	\vdots & \ddots & \vdots & \cdots &
	\vdots & \vdots & \vdots & \vdots & \cdots &
	\vdots & \cdots & \vdots \\
	0 & 0 & 0 & \cdots & -t_{1n} & 0 & 0 & 0 & \cdots & -t_{2n} & 0 & 0 & 0 & \cdots & -t_{3n} & \cdots & -t_{n-1,n} 
\end{bmatrix}, \\
A_{2}&=\begin{bmatrix}
	t_{12} & t_{13} & t_{14} & \cdots & t_{1n} &
	0 & 0 & 0 & \cdots & 0 & 0 & 0 & 0 & \cdots & 0 & \cdots & 0 \\
	0 & 0 & 0 & \cdots & 0 & t_{23} & t_{24} & t_{25} & \cdots & t_{2n} & 0 & 0 & 0 & \cdots & 0 & \cdots & 0 \\
	0 & 0 & 0 & \cdots & 0 & 0 & 0 & 0 & \cdots & 0 & t_{34} & t_{35} & t_{36} & \cdots & t_{3n} & \cdots & 0 \\
	\vdots & \vdots & \vdots & \cdots & \vdots & \vdots & \vdots & \vdots & \cdots & \vdots & \vdots & \vdots & \vdots & \cdots & \vdots & \cdots & \vdots \\
	0 & 0 & 0 & \cdots & 0 & 0 & 0 & 0 & \cdots & 0 & 0 & 0 & 0 & \cdots & 0 & \cdots & 0 
\end{bmatrix},
\end{align*}
and the matrices $B_{1}$ and $B_{2}$ have the same forms as the matrices $A_{1}$ and $A_{2}$ except that the element $t_{ij}$ is replaced by $r_{ij}$.
Setting
\begin{equation}
    E=\begin{bmatrix}
	A_{1} & A_{2} \\ B_{1} & B_{2}
\end{bmatrix},
Z=\begin{bmatrix}
	EL^{-1} , I_{2n}
\end{bmatrix},
\label{Z}
\end{equation}
neglecting the second order term in (\ref{eq:sec3.10}), and using (\ref{LNI1}), we have
\begin{align*}
    \begin{bmatrix}
	\Delta \lambda^t \\ \Delta \lambda^r
\end{bmatrix}=EL^{-1}\begin{bmatrix}
	f \\ g
\end{bmatrix}+\begin{bmatrix}
	f_{1} \\ g_{1}
\end{bmatrix}
=\begin{bmatrix}
	EL^{-1} , I_{2n}
\end{bmatrix}\begin{bmatrix}
	f \\ g \\
	f_{1} \\  g_{1}
\end{bmatrix}.
\end{align*}
Since 
\begin{align*}
    \left\Vert \begin{bmatrix}
	f \\ g \\
	f_{1} \\  g_{1}
\end{bmatrix} \right\Vert_{2} \leq \left\Vert \begin{bmatrix}
	\Delta A \\ \Delta B
\end{bmatrix} \right\Vert_{F},
\end{align*}
the linear componentwise perturbation bounds of $\Delta \lambda^t$ and $\Delta \lambda^r$ are obtained as follows
\begin{equation}
	\left\vert \begin{bmatrix}
		\Delta \lambda^t \\ \Delta \lambda^r
	\end{bmatrix}_{i}\right\vert \leq \Vert Z_{i,:}\Vert_{2}\left\Vert \begin{bmatrix}
		\Delta A \\ \Delta B
	\end{bmatrix} \right\Vert_{F},
	\label{equ2}
\end{equation} 
where $i=1,2,\dots,2n$ and $\begin{bmatrix}
	\Delta \lambda^t \\ \Delta \lambda^r
\end{bmatrix}_{i}$ stands for the $i$th component of the vector $\begin{bmatrix}
	\Delta \lambda^t \\ \Delta \lambda^r
\end{bmatrix}$.

In summary, we have the following theorem.
\begin{theorem} \label{theorem5.1}
	Let $(A,B)$ and $(\widetilde{A},\widetilde{B})$ have the generalized Schur decompositions in (\ref{eq:sec1.1}) and (\ref{eq:sec2.1}), respectively. 
	Then the linear perturbation bounds for the diagonal elements of $T$ and $R$ given in (\ref{equ2})  hold.
\end{theorem}

Using the linear bounds of $\Delta \lambda^t$ and $\Delta \lambda^r$, we can consider the perturbation analysis of the generalized eigenvalue $\lambda$ of the regular matrix pair $(A,B)$, that is, $A\xi=\lambda B\xi$, where $\xi$ is a right eigenvector corresponding to the generalized eigenvalue $\lambda$. Recall that the eigenvalue $\lambda$ can also be represented as a pair of numbers $\langle \alpha,\beta \rangle$ such that $\beta A\xi=\alpha B\xi$. If $\beta\neq0$, $\lambda=\alpha/\beta$, otherwise $\lambda=\infty$. As mentioned in Section \ref{sec:int and no}, $\alpha$ and $\beta$ can be the diagonal elements of the matrices $T$ and $R$, respectively. 

Stewart and  Sun \cite{stewart1990matrix} defined the following chordal metric to measure the eigenvalue perturbation of a regular matrix pair,
\begin{equation}
	\mathcal{X}(\langle \alpha,\beta \rangle,\langle \widetilde{\alpha},\widetilde{\beta}\rangle)=\frac{\vert \alpha \widetilde{\beta}-\beta\widetilde{\alpha}\vert}{\sqrt{\vert \alpha\vert^2+\vert \beta\vert^2}\sqrt{\vert \widetilde{\alpha}\vert^2+\vert\widetilde{\beta}\vert^2}},
	\label{define}
\end{equation} 
where $\langle \widetilde{\alpha},\widetilde{\beta}\rangle$ is the perturbed eigenvalue,
and presented an upper bound 
\begin{equation}
	\mathcal{X}(\langle \alpha,\beta\rangle,\langle \widetilde{\alpha},\widetilde{\beta}\rangle)\leq
	\frac{\Vert \xi\Vert_{2}\Vert \eta\Vert_{2}}{\sqrt{\vert \alpha\vert^2+\vert \beta\vert^2}}\left\Vert\begin{bmatrix}
		\Delta A,\Delta B
	\end{bmatrix}\right\Vert_{2},
	\label{sun}
\end{equation}
where $\eta$ is the left eigenvector corresponding to the eigenvalue $\langle \alpha,\beta \rangle$. Note that the chordal metric can also be used for the generalized singular value (e.g.,\cite{sunmatrixperturbationanalysis,xu2018explicit}). 

On the basis of (\ref{define}), we have 
\begin{equation*}
	\mathcal{X}(\langle t_{ii},r_{ii} \rangle,\langle \widetilde{t}_{ii},\widetilde{r}_{ii}\rangle)=\frac{\vert t_{ii} \widetilde{r}_{ii}-r_{ii}\widetilde{t}_{ii}\vert}{\sqrt{\vert t_{ii}\vert^2+\vert r_{ii}\vert^2}\sqrt{\vert \widetilde{t}_{ii}\vert^2+\vert\widetilde{r}_{ii}\vert^2}}, \ i=1,2,\dots,n,
\end{equation*} 
which combined with Theorem \ref{theorem5.1}, i.e., 
\begin{align*}
			\vert \widetilde{t}_{ii}-t_{ii} \vert  \leq \Vert Z_{i,:}\Vert_2\left\Vert\begin{bmatrix}
				\Delta A\\ \Delta B\\
			\end{bmatrix}\right\Vert_{F}, \vert \widetilde{r}_{ii}-r_{ii} \vert  \leq\Vert Z_{i+n,:}\Vert_2\left\Vert\begin{bmatrix}
			\Delta A\\ \Delta B\\
		\end{bmatrix}\right\Vert_{F},
		\end{align*}
implies a new bound for generalized eigenvalue with the chordal metric:
	\begin{equation}
		\mathcal{X}(\langle t_{ii},r_{ii}\rangle,\langle  \widetilde{t}_{ii},\widetilde{r}_{ii}\rangle) \leq
		\frac{\left\vert t_{ii}\right\vert r_{i}+ \left\vert r_{ii}\right\vert t_{i}}
		{\sqrt{\left\vert t_{ii}\right\vert^2+\left\vert r_{ii}\right\vert^2}\sqrt{\left\vert t_{ii}\right\vert^2+\left\vert r_{ii}\right\vert^2-2(\vert t_{ii}\vert t_{i}+\vert r_{ii}\vert r_{i})}},
		\label{our}
	\end{equation}
	where 
		\[t_{i}=\Vert Z_{i,:}\Vert_2\left\Vert\begin{bmatrix}
		\Delta A\\ \Delta B\\
	\end{bmatrix}\right\Vert_{F}, r_{i}=\Vert Z_{i+n,:}\Vert_2\left\Vert\begin{bmatrix}
		\Delta A\\ \Delta B\\
	\end{bmatrix}\right\Vert_{F}.\]
This is because 
	\begin{align*}
	    \vert t_{ii}\widetilde{r}_{ii}-r_{ii}\widetilde{t}_{ii}\vert &= \vert t_{ii}(\widetilde{r}_{ii}-r_{ii})-r_{ii}(\widetilde{t}_{ii}-t_{ii}) \vert 
	     \leq \vert t_{ii}(\widetilde{r}_{ii}-r_{ii}) \vert + \vert r_{ii}(\widetilde{t}_{ii}-t_{ii}) \vert \\
	    & \leq \vert t_{ii} \vert \vert\widetilde{r}_{ii}-r_{ii} \vert + \vert r_{ii}\vert \vert \widetilde{t}_{ii}-t_{ii} \vert 
	     \leq \vert t_{ii} \vert r_{i} + \vert r_{ii}\vert t_{i},
	\end{align*}
and
	 \begin{align*}
	     \vert \widetilde{t}_{ii} \vert^2 + \vert \widetilde{r}_{ii} \vert^2 &= \vert \widetilde{t}_{ii}-t_{ii}+t_{ii}\vert^2 + \vert \widetilde{r}_{ii}-r_{ii}+r_{ii}\vert^2 \\
	     & \geq (\vert \widetilde{t}_{ii}-t_{ii}\vert-\vert t_{ii}\vert)^2+ (\vert \widetilde{r}_{ii}-r_{ii}\vert-\vert r_{ii}\vert)^2 \\
	     &= \vert \widetilde{t}_{ii}-t_{ii}\vert^2 + \vert \widetilde{r}_{ii}-r_{ii}\vert^2 +\vert t_{ii}\vert^2+\vert r_{ii}\vert^2 - 2(\vert \widetilde{t}_{ii}-t_{ii}\vert\vert t_{ii}\vert + \vert \widetilde{r}_{ii}-r_{ii}\vert\vert r_{ii}\vert) \\
	     & \geq \vert t_{ii}\vert^2 + \vert r_{ii}\vert^2 - 2(\vert \widetilde{t}_{ii}-t_{ii}\vert\vert t_{ii}\vert + \vert \widetilde{r}_{ii}-r_{ii}\vert\vert r_{ii}\vert) \\
	     & \geq \vert t_{ii}\vert^2 + \vert r_{ii}\vert^2 - 2(\vert t_{ii}\vert t_{i}+\vert r_{ii}\vert r_{i}) > 0,
	 \end{align*} 
	 where the penultimate inequality relies on $\vert \widetilde{t}_{ii}-t_{ii}\vert \leq t_{i}$ and $\vert \widetilde{r}_{ii}-r_{ii}\vert \leq r_{i}$, and the last inequality comes from the assumption that $t_{i}\ll \vert t_{ii}\vert$ and $r_{i}\ll \vert r_{ii}\vert$. 


\begin{remark}
	Compared with the bound (\ref{sun}), our bound (\ref{our}) doesn't contain the information on the left and right generalized eigenvectors. A numerical example in Section \ref{sec:num} shows that our bound can be more accurate than (\ref{sun}) in most cases. 
\end{remark}

Now we consider the condition number of the generalized eigenvalue $\lambda$ (e.g.,\cite{fraysse1998note,higham1998structured,diao2019structured}). A definition similar to the one in \cite{higham1998structured} is first given as follows
\begin{equation}
	{\rm cond}(\lambda):=\lim_{\epsilon\to\ 0}\sup \Big \{ \frac{\vert \Delta \lambda \vert}{\epsilon\vert \lambda\vert}:(A+\Delta A)(\xi+\Delta \xi)=(\lambda +\Delta\lambda)(B+\Delta B)(\xi+\Delta \xi), \left\Vert\begin{bmatrix}
		\Delta A,\Delta B
	\end{bmatrix}\right\Vert_{F}\leq \epsilon \Big \}.
	\label{condition define}
\end{equation}
And, as done in \cite{higham1998structured}, we assume that $\Delta \xi \to 0$ as $\epsilon \to 0$ in order to prevent $\rm cond(\lambda)=\infty$. In the following, we will obtain an explicit expression of ${\rm cond}(\lambda)$. First, note that a fact
$\langle \alpha,\beta \rangle=\langle \eta^HA\xi,\eta^HB\xi \rangle$, which was proved by Stewart and Sun \cite[p.277]{stewart1990matrix}. Thus, we have
\begin{equation}
	\begin{split}
		\tilde{\alpha}=\tilde{\eta}^H\tilde{A}\tilde{\xi} &=(\eta^H+\Delta \eta^H)(A+\Delta A)(\xi+\Delta \xi) \\
		&=\alpha+\eta^H\Delta A\xi+O(\epsilon^2) ,
	\end{split}
	\label{a}
\end{equation} 
where $\tilde{\eta}$ and $\tilde{\xi}$ are the left and right generalized eigenvectors corresponding to the eigenvalue $\langle \widetilde{\alpha},\widetilde{\beta}\rangle$ of the matrix pair $(\tilde{A},\tilde{B})$, respectively.
Similarly, we obtain 
\begin{equation}
	\tilde{\beta}=\beta+\eta^H\Delta B\xi+O(\epsilon^2).
	\label{b}
\end{equation}
Thus, combining the equation $(A+\Delta A)(\xi+\Delta \xi)=(\lambda +\Delta\lambda)(B+\Delta B)(\xi+\Delta \xi)$ with (\ref{a}) and (\ref{b}), implies
\[\begin{split}
	\Delta \lambda &=\frac{\eta^H\Delta A\xi-\lambda \eta^H\Delta B\xi}{\eta^HB\xi}+O(\epsilon^2) \\
	&=\frac{\tilde{\alpha}-\alpha-\lambda (\tilde{\beta}-\beta)}{\beta}+O(\epsilon^2) \\
	&=\frac{\Delta \alpha}{\beta}-\frac{\alpha\Delta \beta}{\beta^2}+O(\epsilon^2),
\end{split}\] 
where $\Delta \alpha=\tilde{\alpha}-\alpha$ and $\Delta \beta=\tilde{\beta}-\beta$ are the  perturbations of $\alpha$ and $\beta$, respectively. 
Alternatively, we can get $\Delta \lambda$ more easily by Taylor expansion, that is,
\begin{align*}
	\Delta \lambda&=\lambda+\Delta \lambda-\lambda=\frac{\alpha+\Delta \alpha}{\beta+\Delta \beta}-\frac{\alpha}{\beta}
	=\frac{\alpha+\Delta \alpha}{\beta}(1-\frac{\Delta \beta}{\beta})-\frac{\alpha}{\beta}+O(\epsilon^2) \\
	&=\frac{\Delta \alpha}{\beta}-\frac{\alpha\Delta \beta}{\beta^2}+O(\epsilon^2).
\end{align*}

Therefore, from (\ref{condition define}), we have 
\begin{equation}
	\frac{\vert \Delta \lambda \vert}{\vert  \lambda \vert}\leq \left\vert \frac{\Delta \alpha}{\alpha}\right\vert+\left\vert \frac{\Delta \beta}{\beta}\right\vert \leq \left(\left\vert \frac{\Delta \alpha}{\alpha}\right\vert+\left\vert \frac{\Delta \beta}{\beta}\right\vert\right)\frac{\epsilon}{\left\Vert\begin{bmatrix}
			\Delta A,\Delta B
		\end{bmatrix}\right\Vert_{F}}.
	\label{cond1}
\end{equation}
Dividing (\ref{cond1}) by $\epsilon$ and considering (\ref{equ2}), the condition number of the $i$th eigenvalue $\lambda_i$ can be obtained: 
\begin{align}
	\frac{\vert \Delta \lambda_i \vert}{\epsilon\vert  \lambda_i \vert}&\leq\left(\left \vert \frac{\Delta \alpha}{\alpha}\right\vert+\left\vert \frac{\Delta \beta}{\beta}\right\vert\right)\frac{1}{\left\Vert\begin{bmatrix}
			\Delta A,\Delta B
		\end{bmatrix}\right\Vert_{F}}\nonumber\\
		&\leq \frac{\Vert Z_{i,:}\Vert_{2}}{\vert t_{ii}\vert}+\frac{\Vert Z_{i+n,:}\Vert_{2}}{\vert r_{ii}\vert},
			\label{cond2}
\end{align}
where $i=1,2,\dots,n$ and $Z$ is given in (\ref{Z}). 

From the above discussions, we conclude the following theorem.
\begin{theorem}
	The condition number of the $i$th generalized eigenvalue $\lambda_i$ of $(A,B)$ can be given in (\ref{cond2}).
\end{theorem}

\begin{remark}
Although the derivation of the condition number (\ref{cond2}) is done using the left and right generalized eigenvectors, the condition number itself doesn't include the information on 
generalized eigenvectors. In addition, it should be pointed out that
 (\ref{cond2}) is actual an upper bound for the true condition number. 
\end{remark}

\begin{remark}
	According to (\ref{sun}), Frayss{\'e} and Toumazou \cite{fraysse1998note} derived a upper bound of  $\Delta \lambda$
	\begin{align*}
	    \vert \Delta \lambda\vert\leq \sqrt{1+\vert \lambda\vert^2}\frac{\Vert \xi\Vert_{2}\Vert \eta\Vert_{2}}{\vert \eta^HB\xi\vert}\left\Vert\begin{bmatrix}
		\Delta A,\Delta B
	\end{bmatrix}\right\Vert_{2}.
	\end{align*}
	Thus, by setting $\Vert[\Delta A,\Delta B]\Vert_{2}\leq \epsilon$, in our notation, we can find a condition number from the results in \cite{stewart1990matrix,fraysse1998note}
	\begin{equation}
		\frac{\vert \Delta \lambda \vert}{\epsilon\vert  \lambda \vert}\leq \frac{\sqrt{1+\vert \lambda\vert^2}}{\vert \lambda \vert}\frac{\Vert \xi\Vert_{2}\Vert \eta\Vert_{2}}{\vert \eta^HB\xi\vert}.
		\label{sun1}
	\end{equation}
	Actually, \eqref{sun1} is also an upper bound for the true condition number. In addition, Higham \cite{higham1998structured} presented the following normwise condition number
	\begin{equation}
		\frac{\vert \Delta \lambda \vert}{\epsilon\vert  \lambda \vert}\leq {\rm normcond}(\lambda)=\frac{\Vert \eta\Vert_2\Vert \xi\Vert_2(\Vert E\Vert_2+\vert \lambda\vert \Vert F\Vert_2)}{\vert\lambda\vert\vert \eta^HB\xi\vert},
		\label{normcond}
	\end{equation}
	by setting $\Vert \Delta A\Vert_2 \leq \epsilon\Vert E \Vert_2,\Vert \Delta B\Vert_2 \leq \epsilon\Vert F\Vert_2$, and the following componentwise condition number
	\begin{equation}
		\frac{\vert \Delta \lambda \vert}{\epsilon\vert  \lambda \vert}\leq {\rm compcond}(\lambda)=\frac{\vert \eta^H\vert E \vert \xi\vert +\vert \lambda \vert \vert \eta^H\vert F \vert \xi \vert}{\vert \lambda \vert \vert \eta^HB\xi\vert},
		\label{compcond}
	\end{equation}
	by setting $\vert \Delta A \vert\leq \epsilon E,\vert \Delta B \vert\leq \epsilon F$. Note that the definitions of the above four condition numbers are different. Specifically, the definitions for the condition numbers (\ref{cond2}),  (\ref{sun1}), (\ref{normcond}), and (\ref{compcond}) are (\ref{condition define}),
		\begin{align*}
		    &\lim_{\epsilon\to\ 0}\sup \Big \{ \frac{\vert \Delta \lambda \vert}{\epsilon\vert \lambda\vert}:(A+\Delta A)(\xi+\Delta \xi)=(\lambda +\Delta\lambda)(B+\Delta B)(\xi+\Delta \xi), \left\Vert\begin{bmatrix}
		\Delta A,\Delta B
	\end{bmatrix}\right\Vert_{2}\leq \epsilon \Big \},\\
		    &\lim_{\epsilon\to\ 0}\sup \Big \{ \frac{\vert \Delta \lambda \vert}{\epsilon\vert \lambda\vert}:(A+\Delta A)(\xi+\Delta \xi)=(\lambda +\Delta\lambda)(B+\Delta B)(\xi+\Delta \xi), \left\Vert
		\Delta A\right\Vert_{2}\leq \epsilon \Vert E\Vert_2,\left\Vert\Delta B
	\right\Vert_{2}\leq \epsilon \Vert F\Vert_2 \Big \},
		\end{align*}
	and	
	\begin{align*}
		    \lim_{\epsilon\to\ 0}\sup \Big \{ \frac{\vert \Delta \lambda \vert}{\epsilon\vert \lambda\vert}:(A+\Delta A)(\xi+\Delta \xi)=(\lambda +\Delta\lambda)(B+\Delta B)(\xi+\Delta \xi), \left\vert
		\Delta A\right\vert\leq \epsilon E,\left\vert\Delta B
	\right\vert\leq \epsilon F \Big \},
		\end{align*} 
		respectively. So, it is difficult to compare the condition numbers (\ref{cond2}), (\ref{sun1}), (\ref{normcond}) and (\ref{compcond}) in theory,  and they are usually not equal either. In Section \ref{sec:num}, we will compare them numerically.
\end{remark}

\section{Linear perturbation bounds for generalized invariant subspace}	
\label{sec:invariant}
We first present the definition of the generalized invariant subspace.
\begin{myDef}[\cite{sunmatrixperturbationanalysis}]\label{definition}
	Let $(A,B)$ be a regular matrix pair. If there are unitary matrices $Y$ and $X=[X_1,X_2]$, where $X_1 \in \mathbb{C}^{n\times p}$, such that
	\begin{align*}
	    Y^HAX=\begin{bmatrix}
		A_{11} & A_{12} \\ 0 & A_{22}
	\end{bmatrix},Y^HBX=\begin{bmatrix}
		B_{11} & B_{12} \\ 0 & B_{22}
	\end{bmatrix},
	\end{align*}
	where $A_{11}, B_{11} \in \mathbb{C}^{p\times p}$, we call the column subspace of $X_1$, i.e., $R(X_1)$, the $p$-dimensional generalized invariant subspace of $(A,B)$.
\end{myDef}
Combining Definition \ref{definition} with the generalized Schur decomposition (\ref{eq:sec1.1}), we find that the perturbation of $p$-dimensional generalized invariant subspace is the perturbation of the span of the first $p$ columns of the matrix $V$. 

We denote by $\mathscr{X}$ the $p$-dimensional unperturbed generalized invariant subspace of $(A,B)$, by $\widetilde{\mathscr{X}}$ the $p$-dimensional perturbed generalized invariant subspace of $(\widetilde{A},\widetilde{B})$, and by $V_{X}$ and $\widetilde{V}_{X}$ the orthonormal bases for $\mathscr{X}$ and $\widetilde{\mathscr{X}}$, respectively. 
From \cite{davis1970rotation}, the distance between $\mathscr{X}$ and $\widetilde{\mathscr{X}}$ is measured by
\begin{align*}
    {\rm{sin}}(\Theta_{\rm {max}}(\mathscr{X}, \widetilde{\mathscr{X}}))=\Vert {V_{X}^{\perp}}^H\widetilde{V}_{X}\Vert_{2},
\end{align*}
where $V_{X}^{\perp}$ is the unitary complement of $V_{X}$. Since 
\begin{align*}
    \widetilde{V}_{X}=V_X+\Delta V_X,  {V_{X}^{\perp}}^H V_{X}=0,
\end{align*}
we have 
\begin{equation}
	{\rm{sin}}(\Theta_{\rm {max}}(\mathscr{X}, \widetilde{\mathscr{X}}))=\Vert {V_{X}^{\perp}}^H\Delta V_X\Vert_{2}.
	\label{sin}
\end{equation}
According to (\ref{sin}), it is easy to find that the perturbation of generalized invariant subspace is related to the strictly lower triangular part of the matrix $\Delta K$ given in (\ref{eq:sec2.6}), that is, the elements $y_l=v_i^H\Delta v_j$ with $l=i+(j-1)n-\frac{j(j+1)}{2}$ and $1\leq j< i\leq n$. 
Hence,
\begin{equation}
	{\rm{sin}}(\Theta_{\rm {max}}(\mathscr{X}, \widetilde{\mathscr{X}}))=\Vert \Delta K_{p+1:n,1:p}\Vert_{2}.
\end{equation}
Then the maximum angle between the perturbed and unperturbed invariant subspace of dimension $p$ is 
\begin{equation}
	\Theta_{\rm {max}}(\mathscr{X}, \widetilde{\mathscr{X}})={\rm arcsin}(\Vert \Delta K_{p+1:n,1:p}\Vert_{2}).
	\label{6.3}
\end{equation}
Note that from (\ref{LNI}), the linear approximation of $y_l$ is
\begin{equation}
	y_{lin,l}=L_{\nu+l,:}^{-1}\begin{bmatrix}
		f\\g\\
	\end{bmatrix},l=1,2,\dots,\nu.
	\label{6.4}
\end{equation}
Thus, letting the elements $y_l$ of the matrix $\Delta K$ be replaced by the linear approximation, we have the linear approximation of $\Delta K$:
\begin{align*}
    \Delta K_{lin}=\begin{bmatrix}
	\star & \star & \star & \cdots & \star \\
	y_{lin,1} & \star & \star & \cdots & \star \\
	y_{lin,2} & y_{lin,n} & \star & \cdots & \star \\
	\vdots & \vdots & \vdots & \ddots &\vdots \\
	y_{lin,n-1} & y_{lin,2n-3} & y_{lin,3n-6} & \cdots & \star \\
\end{bmatrix}\in \mathbb{C}^{n\times n}.
\end{align*}
Here the stars $\star$ are numbers we're not interested in. Further, set
\begin{equation}
    \widetilde{L}=\begin{bmatrix}
	\star & \star & \star & \cdots & \star \\
	L_{\nu+1,:}^{-1} & \star & \star & \cdots & \star \\
	L_{\nu+2,:}^{-1} & L_{\nu+n,:}^{-1} & \star & \cdots & \star \\
	\vdots & \vdots & \vdots & \ddots &\vdots \\
	L_{\nu+n-1,:}^{-1} & L_{\nu+2n-3,:}^{-1} & L_{\nu+3n-6,:}^{-1} & \cdots & \star \\
\end{bmatrix}\in \mathbb{C}^{n\times 2n\nu},
\label{Lwan}
\end{equation}
\begin{align*}
    D_p={\rm diag}\left(\begin{bmatrix}
	f \\g 
\end{bmatrix},\begin{bmatrix}
	f \\g 
\end{bmatrix},\cdots,\begin{bmatrix}
	f \\g 
\end{bmatrix}\right)\in \mathbb{C}^{2p\nu \times p}.
\end{align*}
Thus, considering (\ref{6.4}), we have 
\begin{equation}
	\Delta K_{lin,(p+1:n,1:p)}=\widetilde{L}_{p+1:n,1:2p\nu}D_p.
	\label{linear delta k}
\end{equation}
Combining (\ref{6.3}), (\ref{linear delta k}), and the following result
\begin{align*}
    \Vert D_p\Vert_{2}=\left\Vert \begin{bmatrix}
	f \\ g
\end{bmatrix}\right\Vert_{2}\leq \left\Vert \begin{bmatrix}
	\Delta A \\ \Delta B
\end{bmatrix}\right\Vert_{F},
\end{align*}
we obtain the linear perturbation bound for the $p$-dimensional generalized invariant subspace:
\begin{equation}
	\Theta_{\rm {max}}(\mathscr{X}, \widetilde{\mathscr{X}})\leq {\rm arcsin}\left(\Vert \widetilde{L}_{p+1:n,1:2p\nu}\Vert_{2}\left\Vert \begin{bmatrix}
		\Delta A \\ \Delta B
	\end{bmatrix}\right\Vert_{F}\right).
	\label{perturbation of invariant}
\end{equation}

We summary the above discussions in the following theorem.

\begin{theorem} \label{theorem6.1}
	The linear perturbation bound  for the $p$-dimensional generalized invariant subspace of the  regular  matrix  pair $(A,B)$ given in (\ref{perturbation of invariant}) holds.
\end{theorem}

\begin{remark}
    As done in \cite{petkov2021componentwise}, we can regard the term
\begin{equation}
	{\rm cond}(\Theta)=\Vert \widetilde{L}_{p+1:n,1:2p\nu}\Vert_{2}
	\label{cond of invariant1}
\end{equation}
as a condition number of the $p$-dimensional generalized invariant subspace. 
\end{remark}

\begin{remark}
	Let $T$ and $R$ in (\ref{eq:sec1.1}) be partitioned as
	\begin{align*}
	    T=\begin{bmatrix}
		T_{11} &T_{12} \\ 0 & T_{22} \\
	\end{bmatrix},R=\begin{bmatrix}
		R_{11} &R_{12} \\ 0 & R_{22} \\
	\end{bmatrix},
	\end{align*}
	where $T_{11}$ and $R_{11}\in \mathbb{C}^{p\times p}$. By defining a linear operator as
	\begin{equation}
	    \mathbf{T}(P,Q)=(QT_{11}-T_{22}P,QR_{11}-R_{22}P),
	    \label{linear operator}
	\end{equation}
	and writing its matrix as
	\begin{align*}
	    H=\begin{bmatrix}
		-I_p\otimes T_{22} & T_{11}\otimes I_{n-p}\\
		-I_p\otimes R_{22} & R_{11}\otimes I_{n-p}\\
	\end{bmatrix}\in \mathbb{C}^{2p(n-p)\times 2p(n-p)}, 
	\end{align*}
	Sun \cite{sunmatrixperturbationanalysis} obtained the following bound
	\begin{equation}
		\Vert {\rm{sin}}(\Theta(\mathscr{X}, \widetilde{\mathscr{X}}))\Vert_{F}\leq
		\Vert {\rm{tan}}(\Theta(\mathscr{X}, \widetilde{\mathscr{X}}))\Vert_{F}\leq
		s\left\Vert\begin{bmatrix}
			\Delta A \\ \Delta B
		\end{bmatrix}\right\Vert_F,
		\label{cond of invariant2}
	\end{equation}
	where $s=\Vert S_1\Vert_2$ with $S_1=H_{1:p(n-p),:}^{-1}\in \mathbb{C}^{p(n-p)\times 2p(n-p)}$,
	and regarded $s$ as the condition number for the  $p$-dimensional generalized invariant subspace of $(A,B)$. In addition, according to the linear operator (\ref{linear operator}), Stewart \cite{stewart1973error} defined the function dif as follows:
		\begin{align*}
			{\rm dif}(T_{11},R_{11};T_{22},R_{22})=\mathop{\rm inf}\limits_{\Vert (P,Q)\Vert_{F}=1}\Vert \mathbf{T}(P,Q) \Vert_{F},
		\end{align*}
	where $\Vert (P,Q)\Vert_{F}={\rm max}(\Vert P\Vert_{F},\Vert Q \Vert_{F})$. If the linear operator $\mathbf{T}$ is invertible, then 
	\begin{equation}
		\frac{1}{{\rm dif}(T_{11},R_{11};T_{22},R_{22})}=\Vert \mathbf{T}^{-1} \Vert_2. 
		\label{dif}
	\end{equation}
    From Sun \cite[p.395]{sunmatrixperturbationanalysis}, we know that $s \leq {\rm dif}^{-1}=\Vert \mathbf{T}^{-1} \Vert_2$.    
     In Section \ref{sec:num},
	we will compare (\ref{cond of invariant1}), $s$, and  ${\rm dif}^{-1}$ numerically.
\end{remark}


\section{Nonlinear perturbation bounds}
\label{sec:nonlinear}
In the previous discussions for componentwise perturbation bounds, all the second order terms are neglected. In this part, we try to use numerical method to calculate the perturbation bounds including the second order terms, i.e., the nonlinear perturbation bounds. The numerical method is an iterative method for deriving the solution $\widetilde{x}$ of the equation $T\widetilde{x}=g-\varphi (\widetilde{x})$ (e.g., \cite{stewart1973error}). Specifically, the solution $\widetilde{x}$ can be calculated by the following procedure
\begin{align*} 
	&1.\quad \widetilde{x}_0=0, \\
	&2. \quad {\rm for} \quad i=0,1,2,\cdots \\
	&1) \qquad r_i=g-\varphi (\widetilde{x}_i)-T\widetilde{x}_i ,\\
	&2) \qquad d_i=S^{-1}r_i ,\\
	&3) \qquad \widetilde{x}_{i+1}=\widetilde{x}_i+d_i,
\end{align*}
where $S$ is an approximation  to $T$.
More details can be found in \cite{stewart1973error}. Petkov \cite{petkov2021componentwise} ever used this approach to derive the nonlinear perturbation bounds of the Schur decomposition. In the following, we apply it to first calculate the nonlinear bounds of $x$ and $y$ in (\ref{eq:sec2.9}) and of $\Delta W$ and $\Delta K$ in (\ref{eq:sec2.6}) and then the nonlinear perturbation bounds of the generalized Schur decomposition.

Let $\Delta W=\begin{bmatrix}
	\Delta w_1,\Delta w_2,\cdots,\Delta w_n
\end{bmatrix}$ and $\Delta K=\begin{bmatrix}
	\Delta k_1,\Delta k_2,\cdots,\Delta k_n
\end{bmatrix}$. Then considering (\ref{eq:sec2.6}), we have 
\[\Delta w_j=U^H\Delta u_j, \Delta k_j=V^H\Delta v_j,\]
where $j=1,2,\cdots,n$, which together with the facts that $U$ and $V$ are unitary matrices yield
\begin{equation}
	\Delta u_i^H\Delta u_j=\Delta w_i^H\Delta w_j,
	\Delta v_i^H\Delta v_j=\Delta k_i^H\Delta k_j,
	\label{eq:uwvk}
\end{equation}
where $i,j=1,2,\cdots,n$. 
Thus, combining (\ref{delta w2}), (\ref{delta w3}), (\ref{eq:uwvk}) and the fact that $u_{i}^H\Delta u_{i}=-\frac{\Delta u_{i}^H\Delta u_{i}}{2}$ and $ v_{i}^H\Delta v_{i}=-\frac{\Delta v_{i}^H\Delta v_{i}}{2}$, which are proved in Section \ref{sec:u and v},  we have
\begin{align*}
    \Delta W_{2}
&= \begin{bmatrix}
	\frac{-\Vert \Delta w_1\Vert_{2}^2}{2} & 0 & 0 & \cdots & 0 \\
	0 & \frac{-\Vert \Delta w_2\Vert_{2}^2}{2} & 0 & \cdots & 0 \\
	0 & 0 & \frac{-\Vert \Delta w_3\Vert_{2}^2}{2} & \cdots & 0 \\
	\vdots & \vdots & \vdots & \ddots & \vdots \\
	0 & 0 & 0 & \cdots & \frac{-\Vert \Delta w_n\Vert_{2}^2}{2} \\
\end{bmatrix}, \\
\Delta W_{3}&=
\begin{bmatrix}
	0 & \Delta w_{1}^H\Delta w_{2} & \Delta w_{1}^H\Delta w_{3} & \dots & \Delta w_{1}^H\Delta w_{n} \\
	0 & 0 & \Delta w_{2}^H\Delta w_{3} & \dots & \Delta w_{2}^H\Delta w_{n} \\
	0 & 0 & 0 & \dots & \Delta w_{3}^H\Delta w_{n} \\
	\vdots & \vdots & \vdots & \ddots & \vdots\\
	0 & 0 & 0 & \dots & 0 
\end{bmatrix}.
\end{align*}
Similarly, the matrices $\Delta K_2$ and $\Delta K_3$ have the same forms as the matrices $\Delta W_2$ and $\Delta W_3$ except that $\Delta w_i$ is replaced by $\Delta k_i$. Thus, considering (\ref{eq:sec3.6}), we get
\begin{align*}
	\vert \Delta W\vert&=[\vert \Delta w_1\vert,\vert \Delta w_2\vert,\cdots,\vert \Delta w_n\vert]\leq \vert \Delta W_1\vert+\vert \Delta W_2\vert+\vert \Delta W_3\vert \\
	&=\begin{bmatrix}
		0 & \vert x_{1} \vert & \vert x_{2}\vert & \dots & \vert x_{n-1} \vert \\
		\vert x_{1} \vert & 0 & \vert x_{n} \vert & \dots & \vert x_{2n-3} \vert \\
		\vert x_{2} \vert & \vert x_{n} \vert & 0 & \dots & \vert x_{3n-6} \vert \\
		\vdots & \vdots & \vdots & \ddots & \vdots\\
		\vert x_{n-1} \vert & \vert x_{2n-3} \vert &
		\vert x_{3n-6} \vert & \dots & 0
	\end{bmatrix}+\begin{bmatrix}
		\frac{\Vert \Delta w_1\Vert_{2}^2}{2} & 0 & 0 & \cdots & 0 \\
		0 & \frac{\Vert \Delta w_2\Vert_{2}^2}{2} & 0 & \cdots & 0 \\
		0 & 0 & \frac{\Vert \Delta w_3\Vert_{2}^2}{2} & \cdots & 0 \\
		\vdots & \vdots & \vdots & \ddots & \vdots \\
		0 & 0 & 0 & \cdots & \frac{\Vert \Delta w_n\Vert_{2}^2}{2} \\
	\end{bmatrix} \\
	&+\begin{bmatrix}
		0 & \vert \Delta w_{1}^H\vert \vert \Delta w_{2} \vert&\vert \Delta w_{1}^H\vert\vert \Delta w_{3}\vert & \dots & \vert\Delta w_{1}^H\vert\vert\Delta w_{n}\vert \\
		0 & 0 & \vert \Delta w_{2}^H\vert\vert\Delta w_{3} \vert& \dots & \vert\Delta w_{2}^H\vert\vert\Delta w_{n}\vert \\
		0 & 0 & 0 & \dots & \vert\Delta w_{3}^H\vert\vert\Delta w_{n}\vert \\
		\vdots & \vdots & \vdots & \ddots & \vdots\\
		0 & 0 & 0 & \dots & 0 
	\end{bmatrix}.
\end{align*}
From the corresponding columns of the above inequality, we have 
\begin{equation}\label{w_iterative}
	\begin{cases}
		\vert \Delta w_{1}\vert=\vert \Delta W_{11}\vert+\vert \Delta W_{21}\vert ,\\
		\vert \Delta w_{j}\vert\leq \vert S_j^w\vert^{-1}(\vert \Delta W_{1j}\vert+\vert \Delta W_{2j}\vert) ,\\
		\vert S_j^w\vert=I_n-\begin{bmatrix}
			\vert \Delta w_{1}^H\vert \\ \vert \Delta w_{2}^H\vert \\ \vdots \\ \vert \Delta w_{j-1}^H\vert \\ 0 \\ \vdots \\0
		\end{bmatrix}=\begin{bmatrix}
			e_1^T-\vert \Delta w_{1}^H\vert \\ e_2^T-\vert \Delta w_{2}^H\vert \\ \vdots \\ e_{j-1}^T-\vert \Delta w_{j-1}^H\vert \\ e_j^T \\ \vdots \\e_n^T
		\end{bmatrix},
	\end{cases}
\end{equation}
where $j=2,3,\cdots,n$ and $\Delta W_{1j}$ and  $\Delta W_{2j}$ are the $j$th column of $\Delta W_{1}$ and $\Delta W_{2}$, respectively. 
Similarly, we can also obtain 
\begin{equation}\label{k_iterative}
	\begin{cases}
		\vert \Delta k_{1}\vert=\vert \Delta K_{11}\vert+\vert \Delta K_{21}\vert ,\\
		\vert \Delta k_{j}\vert\leq \vert S_j^k\vert^{-1}(\vert \Delta K_{1j}\vert+\vert \Delta K_{2j}\vert), \\
		\vert S_j^k\vert=I_n-\begin{bmatrix}
			\vert \Delta k_{1}^H\vert \\ \vert \Delta k_{2}^H\vert \\ \vdots \\ \vert \Delta k_{j-1}^H\vert \\ 0 \\ \vdots \\0
		\end{bmatrix}=\begin{bmatrix}
			e_1^T-\vert \Delta k_{1}^H\vert \\ e_2^T-\vert \Delta k_{2}^H\vert \\ \vdots \\ e_{j-1}^T-\vert \Delta k_{j-1}^H\vert \\ e_j^T \\ \vdots \\e_n^T
		\end{bmatrix},
	\end{cases}
\end{equation}
where $j=2,3,\cdots,n$ and $\Delta K_{1j}$ and  $\Delta K_{2j}$ are the $j$th column of $\Delta K_{1}$ and $\Delta K_{2}$, respectively.
Furthermore, from (\ref{eq:uwvk}), the estimates of $\vert \Delta_{l}^x\vert$ and $\vert \Delta_{l}^y\vert$ can be written as
\begin{equation} \label{2th of basic perturbation}
	\begin{cases}
		\vert \Delta_{l}^x\vert=\sum_{k=1}^{j}\vert t_{kj}\vert\vert \Delta w_i^H\vert\vert \Delta w_k\vert+\vert \Delta w_i^H\vert\vert U^HAV\vert\vert \Delta k_j\vert ,\\
		\vert \Delta_{l}^y\vert=\sum_{k=1}^{j}\vert r_{kj}\vert\vert \Delta w_i^H\vert\vert \Delta w_k\vert+\vert \Delta w_i^H\vert\vert U^HBV\vert\vert \Delta k_j\vert,
	\end{cases}
\end{equation}
where $1\leq j<i\leq n$ and $l=i+(j-1)n-\frac{j(j+1)}{2}$. 
Thus, putting (\ref{eq:sec2.9}), (\ref{w_iterative}), (\ref{k_iterative}) and (\ref{2th of basic perturbation}) together, we can use  the aforementioned iterative approach  to obtain the nonlinear bounds of  $x$, $y$, $\Delta W$ and $\Delta K$, i.e.,  $\vert x_{nl}\vert$, $\vert y_{nl}\vert$, $\vert \Delta W_{nl}\vert$ and $\vert \Delta K_{nl}\vert$, respectively. The details are summarized in Algorithm \ref{algorithm}.
\begin{algorithm}[ht] 
	\caption{Nonlinear perturbation bounds} 
	\label{algorithm}
	\hspace*{0.02in} {\bf Input:} 
	maximum number $q$ of iterations, tolerance $\epsilon$ \\
	\hspace*{0.02in} {\bf Output:} 
	nonlinear bounds $\vert x_{nl}\vert, \vert y_{nl}\vert, \vert \Delta W_{nl}\vert$ and $\vert \Delta K_{nl}\vert$. 
	\begin{algorithmic}[1]
		\State Let $\vert x_{nl}\vert=0_{\nu\times 1}$, $\vert y_{nl}\vert=0_{\nu\times 1}$, $\vert \Delta W_{nl}\vert=0_{n\times n}$ and $\vert \Delta K_{nl}\vert=0_{n\times n}$.
		\For{iteration$=1,2,\dots,q$}
		\State Combining $\vert x_{nl}\vert$ and $\vert y_{nl}\vert$ with (\ref{Delta W1}) and (\ref{Delta K1}), we can obtain $\vert \Delta W_1\vert$ and $\vert \Delta K_1\vert$.
		\State Let $\vert \Delta W_2\vert=0_{n\times n}$ and $\vert \Delta K_2\vert=0_{n\times n}$.
		\For{$i=1,2,\dots,n$}
		\State $\vert \Delta W_2\vert_{i,i}=\frac{1}{2}\Vert\vert \Delta W_{nl}\vert_{:,i}\Vert_{2}^2$, 
		$\vert \Delta K_2\vert_{i,i}=\frac{1}{2}\Vert\vert \Delta K_{nl}\vert_{:,i}\Vert_{2}^2$.
		\EndFor
		\State Computing $\vert \Delta W_{nl}\vert_{:,1}=\vert \Delta W_1\vert_{:,1}+\vert \Delta W_2\vert_{:,1}$ and  
		$\vert \Delta K_{nl}\vert_{:,1}=\vert \Delta K_1\vert_{:,1}+\vert \Delta K_2\vert_{:,1}$.
		\State Let $S_w=I_n$ and $S_k=I_n$.
		\For{$i=2,3,\dots,n$}
		\State $S_{w(i-1,:)}=S_{w(i-1,:)}-\vert \Delta W_{nl}\vert_{:,i-1}^H$, \State $S_{k(i-1,:)}=S_{k(i-1,:)}-\vert \Delta K_{nl}\vert_{:,i-1}^H$,
		\State $\vert \Delta W_{nl}\vert_{:,i}=S_w^{-1}(\vert \Delta W_1\vert_{:,i}+\vert \Delta W_2\vert_{:,i})$, 
		\State $\vert \Delta K_{nl}\vert_{:,i}=S_k^{-1}(\vert \Delta K_1\vert_{:,i}+\vert \Delta K_2\vert_{:,i})$.
		\EndFor
		\State Let $\vert\Delta^x\vert=0_{\nu\times 1}$, $\vert\Delta^y\vert=0_{\nu\times 1}$ and $l=0$.
		\For{$j=1,2,\dots,n-1$}
		\For{$i=j+1,j+2,\dots,n$}
		\State $l=l+1$,
		\State $\vert \Delta_{l}^x\vert=\sum_{k=1}^{j}\vert t_{kj}\vert\vert \Delta W_{nl}\vert_{:,i}^H\vert \Delta W_{nl}\vert_{:,k}+\vert \Delta W_{nl}\vert_{:,i}^H\vert U^HAV\vert\vert \Delta K_{nl}\vert_{:,j}$,
		\State $\vert \Delta_{l}^y\vert=\sum_{k=1}^{j}\vert r_{kj}\vert\vert \Delta W_{nl}\vert_{:,i}^H\vert \Delta W_{nl}\vert_{:,k}+\vert \Delta W_{nl}\vert_{:,i}^H\vert U^HBV\vert\vert \Delta K_{nl}\vert_{:,j}$.
		\EndFor
		\EndFor 
		\State Let $x=\vert x_{nl}\vert$ and $y=\vert y_{nl}\vert$.
		\State Computing $\begin{bmatrix}
			\vert x_{nl}\vert \\ \vert y_{nl}\vert
		\end{bmatrix}=\begin{bmatrix}
			x \\y \\
		\end{bmatrix}_{lin}+\vert L^{-1}\vert \begin{bmatrix}
			\vert \Delta^x \vert \\ \vert \Delta^y \vert \\ 
		\end{bmatrix}$.
		\State If $\frac{\left\Vert \begin{bmatrix}
				x \\y \\
			\end{bmatrix}-\begin{bmatrix}
				\vert x_{nl}\vert \\ \vert y_{nl}\vert
			\end{bmatrix}\right\Vert_2 }{\left\Vert \begin{bmatrix}
				x \\y \\
			\end{bmatrix}\right\Vert_2}<\epsilon$ then break.
		\EndFor
	\end{algorithmic}
\end{algorithm}

Based on Algorithm \ref{algorithm}, we can find the nonlinear perturbation bounds for generalized Schur decomposition. Specifically, considering (\ref{eq:sec3.7}), we have the following nonlinear perturbation bounds of $ U$ and $ V$:
\begin{equation}
		\begin{cases}
			\vert \Delta U_{nl} \vert =\vert U\vert\vert \Delta W_{nl}\vert, \\
			\vert \Delta V_{nl} \vert = \vert V\vert\vert \Delta K_{nl}\vert .
		\end{cases}
		\label{nonlinear1}
	\end{equation}
Using (\ref{4.3}) and (\ref{4.4}), we obtain the nonlinear perturbation bounds of $T$ and $R$:
\begin{equation}
		\begin{cases}
			\begin{aligned}
				\vert \Delta T_{nl} \vert ={\rm triu}&(\vert U^H\vert (\vert A\vert+\vert \Delta  A\vert)\vert V\vert \vert \Delta K_{nl}\vert+\vert \Delta W_{nl}^H\vert\vert U^H\vert (\vert A\vert+\vert \Delta  A\vert)\vert V\vert \\
				&+\vert \Delta W_{nl}^H\vert\vert U^H\vert (\vert A\vert+\vert \Delta  A\vert)\vert V\vert\vert \Delta K_{nl}\vert+\vert U^H\vert\vert \Delta  A\vert\vert V\vert), 
			\end{aligned}	 \\
			\begin{aligned}
				\vert \Delta R_{nl} \vert={\rm triu}&(\vert U^H\vert (\vert B\vert+\vert \Delta  B\vert)\vert V\vert \vert \Delta K_{nl}\vert+\vert \Delta W_{nl}^H\vert\vert U^H\vert (\vert B\vert+\vert \Delta  B\vert)\vert V\vert \\
				&+\vert \Delta W_{nl}^H\vert\vert U^H\vert (\vert B\vert+\vert \Delta  B\vert)\vert V\vert\vert \Delta K_{nl}\vert+\vert U^H\vert\vert \Delta  B\vert\vert V\vert).
			\end{aligned}	
		\end{cases}	
		\label{nonlinear2}
	\end{equation}
In addition, considering (\ref{6.3}), we have the nonlinear perturbation bound for the  $p$-dimensional generalized invariant subspace: 
\begin{equation}
		\Theta_{{\rm max},nl}={\rm arcsin}(\Vert \vert\Delta K_{nl}\vert_{p+1:n,1:p}\Vert_{2}).
		\label{nonlinear3}
	\end{equation}
At last, considering (\ref{second order diagonal}), 
the bounds for the second order terms in the diagonal elements of $T$ and $R$  can be obtained
\begin{equation}
		\begin{cases}
			\vert \Delta_i^{dt}\vert=\sum_{k=1}^{j}\vert t_{kj}\vert\vert \Delta W_{nl}\vert_{:,i}^H\vert \Delta W_{nl}\vert_{:,k}+\vert \Delta W_{nl}\vert_{:,i}^H\vert U^HAV\vert\vert \Delta K_{nl}\vert_{:,i}, \\
			\vert \Delta_i^{dr}\vert=\sum_{k=1}^{j}\vert r_{kj}\vert\vert \Delta W_{nl}\vert_{:,i}^H\vert \Delta W_{nl}\vert_{:,k}+\vert \Delta W_{nl}\vert_{:,i}^H\vert U^HBV\vert\vert \Delta K_{nl}\vert_{:,i},
		\end{cases}
		\label{nonlinear4}
	\end{equation}
	\begin{equation}
		\vert d_t\vert=\frac{1}{2}\begin{bmatrix}
			\vert t_{11}\vert(\Vert\vert \Delta W_{nl}\vert_{:,1}\Vert_{2}^2+\Vert\vert \Delta K_{nl}\vert_{:,1}\Vert_{2}^2)\\
			\vert t_{22}\vert(\Vert\vert \Delta W_{nl}\vert_{:,2}\Vert_{2}^2+\Vert\vert \Delta K_{nl}\vert_{:,2}\Vert_{2}^2)\\
			\vdots \\
			\vert t_{nn}\vert(\Vert\vert \Delta W_{nl}\vert_{:,n}\Vert_{2}^2+\Vert\vert \Delta K_{nl}\vert_{:,n}\Vert_{2}^2)\\
		\end{bmatrix},
		\vert d_r\vert=\frac{1}{2}\begin{bmatrix}
			\vert r_{11}\vert(\Vert\vert \Delta W_{nl}\vert_{:,1}\Vert_{2}^2+\Vert\vert \Delta K_{nl}\vert_{:,1}\Vert_{2}^2)\\
			\vert t_{22}\vert(\Vert\vert \Delta W_{nl}\vert_{:,2}\Vert_{2}^2+\Vert\vert \Delta K_{nl}\vert_{:,2}\Vert_{2}^2)\\
			\vdots \\
			\vert r_{nn}\vert(\Vert\vert \Delta W_{nl}\vert_{:,n}\Vert_{2}^2+\Vert\vert \Delta K_{nl}\vert_{:,n}\Vert_{2}^2)\\
		\end{bmatrix},
		\label{nonlinear5}
	\end{equation}
where $i=1,2,\cdots,n$. Then we can obtain the nonlinear perturbation bounds for the diagonal elements of $T$ and $R$.

In summary, we have the following theorem.
\begin{theorem} \label{7.1}
	Let $(A,B)$ and $(\widetilde{A},\widetilde{B})$ have the generalized Schur decompositions in (\ref{eq:sec1.1}) and (\ref{eq:sec2.1}), respectively. 
	Then the nonlinear perturbation bounds for $U$ and $V$, $T$ and $R$, and the $p$-dimensional generalized invariant subspace are given in  (\ref{nonlinear1}), (\ref{nonlinear2}) and (\ref{nonlinear3}) respectively, and the bounds for the second order terms in the diagonal elements of $T$ and $R$ are given in (\ref{nonlinear4}) and (\ref{nonlinear5}).
\end{theorem}

\section{Numerical experiments}
\label{sec:num}
All computations are performed in MATLAB R2018b. We compute the generalized Schur decomposition by the MATLAB function qz  
and set all numerical results to be eight decimal digits unless otherwise specified.

In the specific experiments, we consider the matrix pair 
\begin{align*}
    A=\begin{bmatrix*}[r]
	-5.9 & -2.9 & -4.1 & -1.3 &  -11 \\
   5.1 &  -19.5 &  -6  &  -5.1  &  0 \\
  -5.9 &  -1  &  -2.5 &  1  &   -10.1 \\
  -16  &  4  &  -20 & 0 & -2  \\
   6.5 &  -4.5 & -2.1  & -14.5& -2.1
\end{bmatrix*},
B=\begin{bmatrix*}[r]
	4.5 &  4  &   0.5 &  7.9 & -1.1 \\
   3.9  & 5.5 &  4.9 & -5.5 &  1.4 \\
  -5.9 &  0.1 &  7.4 &  2.9 &  1.1 \\
  -4 &  5.4 & -4.1 &  0.5 &  5.9 \\
   3.9 & -4.5 &  1.1 &  1.1 &  7.9
\end{bmatrix*},
\end{align*}
whose eigenvalues are
\begin{align}
    \lambda_{1}=2.71596442, \lambda_{2}=1.41815971,
    \lambda_{3}=-0.01627420,   \lambda_{4}=-1.24264120, \lambda_{5}=-2.43201588, \notag
\end{align}
and let
\begin{align*}
    \Delta A=10^{-6}\times\begin{bmatrix*}[r]
	3 & -2 & 1 & 4 & -1 \\
    1 & 0 & -5 & -4 & 2\\
    -4 & 9 & -3 & 6 & 1\\
    -1 & 0 & 3 & 5 & 4\\
    7 & 3 & 2 & 8 & -2
\end{bmatrix*}, 
\Delta B=10^{-6}\times\begin{bmatrix*}[r]
   1 & 2 & 3 & -2 & 6\\
  -2 & 7 & -4 & 3 & 0\\
   8  & -2 & 4 & 5 & 1\\
   -6 & 0 & -3 & 1 & 8\\
    -5 & -2 & 3 & -9 & 1
\end{bmatrix*}.
\end{align*}

\subsection{On the linear bounds of $x$ and $y$}
By (\ref{eq:sec2.10}), the linear bounds of $x$ and $y$ are given by
\begin{align}
    \vert x_{lin}\vert=10^{-5}\times\begin{bmatrix}
	0.17419294 \\ 0.40018945 \\
	0.10935913\\ 0.12034032 \\
	\overline{0.14269152} \\
	0.12535076 \\0.40248315\\
	\overline{0.21364869} \\ 0.20105311  \\
	\overline{0.38402942}
\end{bmatrix},
\vert y_{lin}\vert=10^{-5}\times\begin{bmatrix}
	0.18192223\\ 0.65324092\\
	0.31267546\\ 0.21624433 \\
	\overline{0.20683921} \\
	0.32197838 \\0.65547555\\
	\overline{ 0.36200403} \\ 0.21289352 \\
	\overline{0.24169720}
\end{bmatrix}, \notag
\end{align}
which compared with the exact absolute value of $x$ and $y$ 
\begin{align}
    \vert x\vert =10^{-5}\times\begin{bmatrix}
	0.00074295\\ 0.12082488\\
	0.03316662\\ 0.02389607 \\
	\overline{0.00734417} \\
	0.02065558 \\0.03259031\\
	\overline{ 0.01503943} \\ 0.01483496 \\
	\overline{0.03244088}
\end{bmatrix},
\vert y\vert =10^{-5}\times\begin{bmatrix}
	0.04106670\\ 0.19483908\\
	0.06753171 \\0.04902623 \\
	\overline{0.04312303} \\
	0.00998254 \\ 0.02704853 \\
	\overline{ 0.05224751}\\
	0.04857184\\ \overline{ 0.05275304}
\end{bmatrix}, \notag
\end{align}
implies that the linear bounds of $x$ and $y$ are mostly an order 
of magnitude larger than the corresponding exact values.  

\subsection{On the linear perturbation bounds of $U$ and $V$}
Based on (\ref{eq:sec3.7}),  the linear perturbation bounds of $U$ and $V$ are 
\begin{align*}
    \vert\widehat{\Delta U}\vert=10^{-5}\times \begin{bmatrix}
	0.38464875 &  0.27246825 &  0.29945131 &  0.26729935 &  0.55913832 \\
    0.26374593 &  0.33476081 &  0.38962465 &  0.38790461 &  0.38347382 \\
    0.38583882 &  0.27766067  & 0.29551131 &  0.25619419 &  0.53648226 \\
    0.24587855 &  0.21005819 &  0.38563454 &  0.23109513 &  0.34735174 \\
    0.32841755 &  0.44142030 &  0.35890629 &  0.46268755 &  0.27965613
\end{bmatrix}, \\
\vert\widehat{\Delta V}\vert=10^{-5}\times \begin{bmatrix}
	0.38424572 &  0.41249836 &  0.74084379 &   0.35938612 &  0.50724759 \\
    0.36311177 &  0.54870742 &  0.64496767 &   0.50670545 &  0.59384236 \\
    0.58005598 &  0.38020000 &  0.62000334 &   0.49848125 &  0.66018186 \\
    0.65402225 &  0.71478018 &  0.49806233 &   0.51246907 &  0.36743149 \\
    0.68978250 &  0.56860387 &  0.39610912 &   0.55848956 &  0.56426987
\end{bmatrix}.
\end{align*}
While the exact perturbations of $U$ and $V$ are 
\begin{align*}
    \vert \Delta U\vert=10^{-5}\times \begin{bmatrix}
	0.03271947 &  0.01004977 &  0.00134424 &  0.02137930 &  0.03189687  \\
    0.00412770 &  0.01651989 &  0.03445908 &  0.00637473 &  0.01312251  \\
    0.08702525 &  0.01759884 &  0.02623930 &  0.01757783 &  0.01775241  \\
    0.05253147 &  0.00445104 &  0.10172491 &  0.02620602  & 0.03710982 \\
    0.06963978  & 0.02897990 &  0.05360180 &  0.03621912 &  0.00487242
\end{bmatrix}, \\
\vert \Delta V\vert=10^{-5}\times\begin{bmatrix}
	0.08740596 &  0.01411635 &  0.17705352 &  0.03493969 &  0.01064313 \\
    0.05844298 &  0.03364732 &  0.07484455 &  0.00846659 &  0.02244902 \\
    0.10897767 &  0.04577163 &  0.08868471 &  0.06560471 &  0.04119389 \\
    0.11657303 &  0.02633445 &  0.00605251 &  0.02382701 &  0.05804865 \\
    0.10045922  & 0.01604350 &  0.00772145 &  0.06331427 &  0.05092626
\end{bmatrix}.
\end{align*}
Comparing $\vert\widehat{\Delta U}\vert$ and $\vert\widehat{\Delta V}\vert$ with $\vert\Delta U\vert$ and $\vert\Delta V\vert$, respectively, we find that the elements of the former are mostly an order of magnitude larger than the latter.

\subsection{On the linear perturbation bounds of $T$ and $R$}
Using (\ref{4.3}) and (\ref{4.4}), the linear  perturbation bounds of $T$ and $R$ are
\begin{align*}
    \vert\widehat{\Delta T}\vert&=10^{-3}\times\begin{bmatrix}
 0.53691509 &  0.51512020 &  0.59550180 &  0.52017664  & 0.52944730 \\
 0  & 0.53438475 &  0.58334135 &  0.53135660 &  0.54212988 \\
 0     &              0 &  0.57265449 &  0.55323306 &  0.52551356 \\
 0      &             0   &                 0 &  0.40707800 &  0.38956929 \\
 0     &              0   &                0   &                0 &  0.48359999
\end{bmatrix}, \\
\vert\widehat{\Delta R}\vert&=10^{-3}\times \begin{bmatrix}
 0.32986610 &  0.31603168 &  0.33529134 &  0.31919222  & 0.31500156 \\
 0 &  0.34828626 &  0.36523305 &  0.33561782 &  0.33582967 \\
 0    &               0 &  0.37103687 &  0.34666968 &  0.34254328 \\
 0     &              0        &           0 &  0.27212787 &  0.26436216 \\
 0       &            0        &           0          &         0 &   0.33100323
\end{bmatrix}.
\end{align*}
And the exact perturbation of $T$ and $R$ are
\begin{align*}
    \vert\Delta T\vert=10^{-4}\times\begin{bmatrix}
	0.08915691  &  0.13322844  &  0.35968307   & 0.09510800  &  0.11156715 \\
    0  &  0.06238591 &   0.09816516  &  0.09494656  &  0.03412775 \\
    0      &              0   & 0.04598561  &  0.09234727   & 0.05254595 \\
    0       &             0      &              0   & 0.00708673  &  0.05780197 \\
    0      &              0      &              0      &              0  &  0.02090659
\end{bmatrix}, \\
\vert\Delta R\vert=10^{-4}\times\begin{bmatrix}
	0.03110455 &  0.05821619 &  0.09853095 &  0.02211806 &  0.00590593 \\ 
    0 &  0.05622772 &  0.01862746 &  0.00221196 &  0.00571853 \\
    0   &                0 &  0.03168093 &  0.02843183 &  0.00131302 \\
    0   &                0  &                 0 &  0.04257788 &  0.11412344 \\
    0    &               0  &                 0  &                 0 &  0.00949379
\end{bmatrix}.
\end{align*}
Comparing $\vert\widehat{\Delta T}\vert$ and $\vert\widehat{\Delta R}\vert$ with $\vert\Delta T\vert$ and $\vert\Delta R\vert$, respectively, we can see that the elements of the former are mostly two orders 
of magnitude larger than the latter. Hence, the bounds for $\Delta T$ and $\Delta R$ are a little loose.

\subsection{On the linear perturbation for diagonal elements of $T$ and $R$}
Considering (\ref{equ2}), the linear perturbation bounds of diagonal elements of $T$ and $R$ are 
\begin{align*}
    \vert\widehat{\Delta \lambda^{t}}\vert=10^{-4}\times\begin{bmatrix}
   0.41987455\\
   0.34753030\\
   0.35436582\\
   0.29611809\\
   0.30903830
\end{bmatrix},
\vert\widehat{\Delta \lambda^{r}}\vert=10^{-4}\times\begin{bmatrix}
   0.29619847\\
   0.29551746\\
   0.29574977\\
   0.29557107\\
   0.29573185
\end{bmatrix},
\end{align*}
which compared with the exact absolute value of vectors $\Delta \lambda^{t}$ and $\Delta \lambda^{r}$
\begin{align*}
    \vert\Delta \lambda^{t}\vert=10^{-5}\times\begin{bmatrix}
   0.89156905\\
   0.62385906\\
   0.45985610\\
   0.07086728\\
   0.20906590
\end{bmatrix},
\vert\Delta \lambda^{r}\vert=10^{-5}\times\begin{bmatrix}
   0.31104545\\
   0.56227722\\
   0.31680932\\
   0.42577879\\
   0.09493786
\end{bmatrix},
\end{align*}
implies that the linear perturbation bounds of $T$ and $R$ are an order of magnitude larger than the corresponding exact perturbation bounds.	

Combining with (\ref{define}),(\ref{sun})  and (\ref{our}), we denote the chordal metric by $C$, Stewart and Sun's upper bound by $C_{1}$ {and} our upper bound by $C_{2}$, {respectively.} 
In Table \ref{tab:table1}, we provide the corresponding results. For simplicity, we only keep four decimal digits in Table \ref{tab:table1}, from which, it is easy to see that $C_2$ is closer to $C$ than $C_1$ in most cases.
 That is, our bound is usually a little tighter. 
\begin{table}[H]
    \caption{Comparison of different bounds for generalized eigenvalues} \label{tab:table1}
    {\begin{tabular}{@{}cccc@{}}
    		\toprule
    		eigenvalue & $C$ & $C_{1}$ & $C_{2}$  \\
    		\midrule
    		$\langle 26.8689,9.8929\rangle$ & 0.0056e-06 & 0.1531e-05 &  0.1477e-05
    		\\
    		$\langle 14.3158,10.0946\rangle$ & 0.2991e-06 & 0.2545e-05 &   0.2522e-05
    		\\
    		$\langle -0.1630,10.0143\rangle$ & 0.0777e-06 & 0.3618e-05 &  0.3586e-05
    		\\
    		$\langle -12.4136,9.9897\rangle$ & 0.0358e-06 &  0.2296e-05&  0.2610e-05
    		\\
    		$\langle -24.0902,9.9054\rangle$ & 0.2907e-06 & 0.1853e-05 & 0.1501e-05 \\
    		\bottomrule
    \end{tabular}}
\end{table}

We denote our condition number by $\rm cond1$ (see (\ref{cond2})), Stewart and Sun's condition number by $\rm cond2$ (see (\ref{sun1})), Higham's normwise condition number by $\rm cond3$ (see (\ref{normcond})) and componentwise condition number by $\rm cond4$ (see (\ref{compcond})). In Table \ref{tab:table2}, we present the corresponding results for different eigenvalue $\lambda$. It should be pointed out that for Higham's condition numbers, we let $E=10^6 \vert\Delta A\vert$ and $F=10^6\vert \Delta B \vert$.
From Table \ref{tab:table2}, we find that our condition number $\rm cond1$ is very close to Stewart and Sun's condition number $\rm cond2$, and both of them are smaller than Higham's condition numbers  $\rm cond3$ and  $\rm cond4$. {It should be clarified here that we can't guarantee the condition number $\rm cond1$ is the smallest one. In fact, the main advantage of our condition number is that it can avoid computing the generalized eigenvectors.} 
\begin{table}[H]
	\caption{Comparison of different condition numbers for generalized eigenvalues} 	\label{tab:table2}
	{\begin{tabular}{@{}ccccc@{}}
			\toprule
			eigenvalue & $\rm cond1$ & $\rm cond2$ & $\rm cond3$ & $\rm cond4$  \\
			\midrule
			2.71596442 & 0.15422147 & 0.11254005 &  2.09032557 & 2.12148832
			\\
			1.41815971 & 0.18293612 & 0.12837272 &   2.58716541& 2.33175429
			\\
			-0.01627420 & 6.24938540 & 6.15018647 &  90.27311163 & 68.67522301
			\\
			-1.24264120 & 0.18445129 & 0.13082771 &  2.66006855 & 2.29093661
			\\
			-2.43201588 & 0.14979764 & 0.11252811 &  2.12456014 & 1.98238882  \\
			\bottomrule	
	\end{tabular}}
\end{table}

\subsection{On the linear perturbation bound of generalized invariant subspace}
Denoting the linear perturbation bound of generalized invariant subspace (\ref{perturbation of invariant}) by $\Theta_{{\rm max},lin}$ and the exact bounds by $\Theta_{\rm {max}}$, we obtain the values of ${\rm cond}(\Theta)$ (see (\ref{cond of invariant1})), $s$ (see (\ref{cond of invariant2})), ${\rm dif}^{-1}$ (see (\ref{dif})),  $\Theta_{{\rm max},lin}$ and $\Theta_{\rm {max}}$ in Table \ref{tab:table3} for different values of dimension $p$. The numerical results in this table show that
our condition numbers ${\rm cond}(\Theta)$ are close to the corresponding values $s$. Especially, when the dimension $p=1$, ${\rm cond}(\Theta)=s$.  And, both ${\rm cond}(\Theta)$ and $s$ are smaller than the quantity ${\rm dif}^{-1}$. 
For the linear perturbation bounds of generalized invariant subspace, the ones of the first two dimensions are tighter than those of the last two.
\begin{table}[H]
	\caption{Linear perturbation bounds for generalized invariant subspace}
	{\begin{tabular}{@{}cccccc@{}}
			\toprule
			$p$ & ${\rm cond}(\Theta)$ & $s$ & ${\rm dif}^{-1}$ & $\Theta_{{\rm max},lin}$ & $\Theta_{\rm {max}}$ \\
			\midrule
			1 & 0.22248707 & 0.22248707 &0.26021577 & 0.65737304e-05 & 0.21590008e-05
			\\
			2 & 0.23633828 & 0.22238147 & 0.25898751 &  0.69829862e-05 & 0.21759910e-05
			\\
			3 & 0.24462325 & 0.22391967 &0.26235478 & 0.72277789e-05 & 0.08860907e-05
			\\
			4 & 0.25778842
			&  0.22389717 &0.26240116 & 0.76167643e-05 & 0.09097967e-05
			\\
			\bottomrule	
	\end{tabular}}
	\label{tab:table3}
\end{table}

\subsection{On the nonlinear perturbation bounds}
The nonlinear bounds for the basic perturbation vectors $x$ and $y$ and the nonlinear perturbation bounds for the matrices $U$, $V$, $T$ and $R$, the diagonal elements of $T$ and $R$, and the generalized invariant subspace are shown below.
\begin{align*}
    &\vert x_{nl}\vert=10^{-5}\times \begin{bmatrix}
	0.17419716 \\ 0.40019720 \\
	0.10936250\\ 0.12034487 \\
	\overline{0.14269725} \\
	0.12535607 \\0.40249318\\
	\overline{0.21365725} \\ 0.20106171 \\
	\overline{0.38403956}
\end{bmatrix},
\vert y_{nl}\vert=10^{-5}\times \begin{bmatrix}
	0.18192650\\ 0.65325071\\
	0.31268299\\ 0.21625081\\
	\overline{0.20684619} \\
	0.32198895 \\0.65548856\\
	\overline{0.36201589} \\ 0.21290221 \\
	\overline{0.24170383}
\end{bmatrix},  \\
&\vert\Delta U_{nl}\vert=10^{-5}\times \begin{bmatrix}
0.38465809 &  0.27247827 &  0.29946417 &  0.26731138 &  0.55915736 \\
0.26375342 &  0.33477088 &  0.38963896 &  0.38791925 &  0.38348753 \\
0.38584808 &  0.27767090 &  0.29552359 &  0.25620577 &  0.53650104 \\
0.24588455 &  0.21006576 &  0.38564508 &  0.23110633 &  0.34736607 \\
0.32842632 &  0.44143317 &  0.35891934 &  0.46270372 &  0.27966863
\end{bmatrix}, \\
&\vert\Delta V_{nl}\vert=10^{-5}\times\begin{bmatrix}
  0.38425598 &  0.41251359 &  0.74086336  & 0.35940094 &  0.50726694 \\
  0.36312214 &  0.54872358 &  0.64498807 &  0.50672427 &  0.59386129 \\
  0.58006851 &  0.38021483 &  0.62002248 &  0.49850195 &  0.66020537 \\
  0.65403630 &  0.71479931 &  0.49807915 &  0.51248822 &  0.36744921 \\
  0.68979642 &  0.56862018 &  0.39612573 &  0.55851131 &  0.56429053
\end{bmatrix}, \\
&\vert\Delta T_{nl}\vert=10^{-3}\times\begin{bmatrix}
	0.53692793 &  0.51513549 &  0.59551860  & 0.52019392 &  0.52946463 \\
    0  & 0.53440164 &  0.58335942 &  0.53137568 &  0.54214841 \\
    0     &              0 &  0.57267282 &  0.55325312 &  0.52553250 \\
    0   &                0     &              0 &  0.40709403 &  0.38958362 \\
    0    &               0     &              0       &            0 &  0.48361764
\end{bmatrix}, \\
&\vert\Delta R_{nl}\vert=10^{-3}\times\begin{bmatrix}
	0.32987371 & 0.31604052 &  0.33530091 &  0.31920237 &  0.31501149 \\
    0  & 0.34829685 &  0.36524427 &  0.33562960 &  0.33584123 \\
    0  &                 0  & 0.37104867 &  0.34668223 &  0.34255543 \\
    0    &               0       &            0 &  0.27213833 &  0.26437218 \\
    0      &             0     &              0      &             0 &  0.33101511
\end{bmatrix}, \\
&\vert\Delta \lambda^{t}_{nl}\vert=10^{-4}\times\begin{bmatrix}
   0.41989720\\
   0.34755180\\
   0.35438670\\
   0.29612336\\
   0.30905865
\end{bmatrix},
\vert\Delta \lambda^{r}_{nl}\vert=10^{-4}\times\begin{bmatrix}
   0.29620829\\
   0.29552763\\
   0.29576128\\
   0.29557888\\
   0.29574460
\end{bmatrix},
\Theta_{{\rm max},nl} =10^{-5}\times\begin{bmatrix}
   0.77741161\\
   0.97532566\\
   0.89176900\\
   0.76169400
\end{bmatrix}.
\end{align*}
We can easily see that the above nonlinear bounds are a little larger than their linear counterparts.

\subsection{On the case  $B=I_n$}
We use the example from \cite{petkov2021componentwise} with $B=I_5$ to compare our results with the corresponding ones for the Schur decomposition from \cite{petkov2021componentwise}. Specifically, let 
	\begin{align*}
		A=\begin{bmatrix*}[r]
			2.5 & -0.4 & 1 & -0.1 & -0.9 \\
			2 & 1.1 & 0 & -0.1 &  -1.9  \\
			2.5 & 1.5 & -1 &  -1 & -3\\
			1.5 &  -1.4 & 2 & 1.9 & -0.9 \\
			-0.5 & -1.4 & 2 & 0.9 & 2.1
		\end{bmatrix*},
		B=I_5,
	\end{align*}
whose eigenvalues are
\begin{align}
	\lambda_{1}=2+i, \lambda_{2}=2-i,\lambda_{3}=1.1,   \lambda_{4}=1, \lambda_{5}=0.5, \notag
\end{align}
and let
\begin{align*}
	\Delta A=10^{-6}\times\begin{bmatrix*}[r]
		-3 & 1 & 7 & -4 & 1\\
		6 & 0 & 4 & 2 & 9 \\
		-3& -2 & 7 & 1 & -5 \\
		8& 6 & -9 & -3 & 4 \\
		7 & 4 & -3 & 2 & 6
	\end{bmatrix*}, \Delta B=0_{5\times 5}.
\end{align*}
The numerical results are reported in Tables \ref{tab:table4}, \ref{tab:table5} and  \ref{tab:table6},  where 
the results for the Schur decomposition are labeled by the subscript $P$. 
From these tables, we find that our bounds or condition numbers are worse than the corresponding ones for the Schur decomposition. Comfortingly, except for the bounds for invariant subspace, the orders of magnitude of our results and the corresponding ones for the Schur decomposition are the same. The main reason for the difference may be that the matrix $L$ defined in (\ref{L}) is different from the matrix $M$ given in \cite{petkov2021componentwise}. For this example, they can be written in the following form: 
\begin{align*}
    &L=\begin{bmatrix}
        -M_{{\rm slvec}}(T^T\otimes I_5)M_{{\rm slvec}}^T & M_{{\rm slvec}}(I_5\otimes T)M_{{\rm slvec}}^T \\
        -I_{10} & I_{10}
    \end{bmatrix}\in \mathbb{C}^{20\times 20},\\
    &M=M_{{\rm slvec}}(I_5\otimes T-T^T\otimes I_5)M_{{\rm slvec}}^T \in \mathbb{C}^{10\times 10}.
\end{align*}
As a result, their inverses are very different, which leads to the final bounds and condition numbers are also different. We find that this phenomenon also appears in \cite[p.294]{stewart1990matrix}.
\begin{table}[H]
	\caption{Comparison of bounds for basic perturbation vector $x$} 	
	\label{tab:table4}
	{\begin{tabular}{@{}cccc@{}}
			\toprule
			$\vert x_{lin} \vert$ & $\vert x_{lin,P}\vert$ & $\vert x_{nl} \vert$ & $\vert x_{nl,P} \vert$   \\
			\midrule
			0.00009686 & 0.00003954 & 0.00009791 & 0.00003972
			\\
			0.00003153 & 0.00001956 & 0.00003213 &  0.00001968 
			\\
			0.00003657 & 0.00002073 & 0.00003708 & 0.00002085 
			\\
			0.00001545 & 0.00001382 & 0.00001549 & 0.00001386
			\\
			0.00006733 & 0.00004203 & 0.00006966 & 0.00004258  \\
			0.00007867 & 0.00004553 & 0.00008057 & 0.00004601   \\
			0.00002619 & 0.00002343 & 0.00002631 & 0.00002350\\
			0.00231591 & 0.00159306 & 0.00238816 & 0.00161609  \\
			0.00009954 & 0.00008903 & 0.00010140 & 0.00008985 \\
			0.00008418 & 0.00007530 & 0.00008656 & 0.00007634  \\
			\bottomrule	
	\end{tabular}}
\end{table}
\begin{table}[H]
	\caption{Comparison of condition numbers and bounds for eigenvalue} 	
	\label{tab:table5}
	{\begin{tabular}{@{}cccccccc@{}}
			\toprule
			$i$ & $\lambda_i$ & $\rm cond1$ & ${\rm cond}_P$ & $\vert\widehat{\Delta \lambda^{t}}\vert$ & $\vert\widehat{\Delta \lambda^{t}}\vert_P$ & $\vert \Delta \lambda^{t}_{nl}\vert$ & $\vert \Delta \lambda^{t}_{nl,P}\vert$  \\
			\midrule
			1 & $2+i$ & 5.70106371 & 4.38748219 & 0.26195515e-03 & 0.10933549e-03 & 0.26206989e-03 & 0.10935099e-03
			\\
			2 & $2-i$ & 5.70106371 & 4.38748219 & 0.26195515e-03 &  0.10933549e-03 & 0.26215081e-03 & 0.10937285e-03
			\\
			3 & 1.1 & 5.57354954 & 3.46410162 & 0.12536949e-03 & 0.08632497e-03 & 0.14482223e-03 & 0.09232970e-03
			\\
			4 & 1 & 5.79583152 & 3.46410162  &  0.11951151e-03 & 0.08632497e-03 & 0.13826436e-03 & 0.09224436e-03
			\\
			5 & 0.5 & 15.10673598 & 6.32455532 & 0.17576902e-03 & 0.15760711e-03  & 0.17633044e-03 & 0.15793190e-03\\
			\bottomrule	
	\end{tabular}}
\end{table}
\begin{table}[H]
	\caption{Comparison of condition numbers and bounds for invariant subspace} 	
	\label{tab:table6}
	{\begin{tabular}{@{}ccccccc@{}}
			\toprule
			$p$ & ${\rm cond}(\Theta)$ & ${\rm cond}(\Theta)_P$ & $\Theta_{{\rm max},lin}$ & $\Theta_{{\rm max},lin,P}$ & $\Theta_{{\rm max},nl}$ & $\Theta_{{\rm max},nl,P}$ \\
			\midrule
			1 & 4.28902636 &  1.75098767 & 0.00010688 & 0.00004363 & 0.00012554 & 0.00005076
			\\
			2 & 4.44048306 &  2.24036494 & 0.00011066 &  0.00005582 & 0.00014758& 0.00007361
			\\
			3 & 93.20159245 &  13.65505303 & 0.00232257 & 0.00034028 & 0.00239584& 0.00034885
			\\
			4 & 6.11926717 &  4.80463834  &  0.00015249 & 0.00011973 & 0.00015593& 0.00012019
			\\
			\bottomrule	
	\end{tabular}}
\end{table}

\noindent {\bfseries Publisher’s Note} Springer Nature remains neutral with regard to jurisdictional claims in published maps and institutional affiliations.

\end{document}